\documentclass[11pt]{article}
\usepackage{amsmath,amsthm,amsfonts,amssymb,amscd, amsxtra,color}
\usepackage{enumerate}

\usepackage{color}
\usepackage{float}
\usepackage{textcomp}
\usepackage{amsthm}
\usepackage{amsmath}
\usepackage{amssymb}

\usepackage{amsmath,amsthm,amsfonts,amssymb,amscd, amsxtra,color}
\usepackage[active]{srcltx}
\usepackage{cite}
\usepackage{multirow}

\usepackage{amssymb,amsfonts,amstext,amsmath}
\usepackage{changes}
\usepackage{mdframed}

\usepackage{colortbl,hhline,color,soul,url}

\usepackage{changes}

\usepackage{mathtools}

\theoremstyle{plain}
\newtheorem{assumption}{Assumption}
\newtheorem{theorem}{Theorem}
\newtheorem{lemma}{Lemma}
\newtheorem{definition}{Definition}
\newtheorem{corollary}{Corollary}
\newtheorem{proposition}{Proposition}

\newtheorem{remark}{Remark}

\DeclareMathOperator{\grad}{grad}  
\DeclareMathOperator{\hess}{Hess}

\DeclareMathOperator{\dt}{dt}  
\DeclareMathOperator{\de}{d}  
\DeclareMathOperator{\argmin}{argmin} 
\DeclareMathOperator{\op}{op}
  
\DeclareMathOperator{\T}{T}   
\DeclareMathOperator{\p}{P}  
\DeclareMathOperator{\B}{B}  
\DeclareMathOperator{\R}{R} 
\DeclareMathOperator{\mR}{\mathbb{R}} 
\DeclareMathOperator{\DR}{DR}

\DeclareMathOperator{\A}{A}  
\DeclareMathOperator{\cM}{\cal M}  
\DeclareMathOperator{\low}{low}

\begin{document}


\title{An Adaptive Cubic Regularization quasi-Newton Method on Riemannian Manifolds}

\author{ Mauricio S. Louzeiro\thanks{School of Computer Science and Technology, Dongguan University of Technology, Dongguan, Guangdong, China and  IME/UFG, Avenida Esperan\c{c}a, s/n, Campus Samambaia, CEP 74690-900, Goi\^ania, GO, Brazil  (Email: {\tt mauriciolouzeiro@ufg.br}). The work of this author was partially supported by National Natural Science Foundation of China (No. 12171087). }
\and
Gilson N. Silva\thanks{Departamento de Matem\'atica, Universidade Federal do Piau\'i­, Teresina, Piau\'i­, 64049-550, Brazil (Email: {\tt
      gilson.silva@ufpi.edu.br}). Thork of this author was partially supported  by
     CNPq, Brazil (401864/2022-7 and 306593/2022-0).}
 \and
 Jinyun Yuan\thanks{School of Computer Science and Technology, Dongguan University of Technology, Dongguan, Guangdong, China (Email: {\tt yuanjy@gmail.com}). The work of this author was partially supported by National Natural Science Foundation of China (No. 12171087), Dongguan University of Technology, China (221110093
 ), and Shanghai Municipal Science and Technology Commission, China (23WZ2501400).}
 \and
Daoping Zhang \thanks{School of Mathematical Sciences and LPMC, Nankai University, Tianjin 300071, China (Email: {\tt
      daopingzhang@nankai.edu.cn}). The work of this author was supported by National Natural Science Foundation of China (No. 12201320) and the Fundamental Research Funds for the Central Universities, Nankai University (No. 63221039 and 63231144).}
}

\maketitle

\begin{abstract}

 A quasi-Newton method with cubic regularization is designed  for solving Riemannian unconstrained nonconvex optimization problems. The proposed algorithm is fully adaptive with at most ${\cal O} (\epsilon_g^{-3/2})$ iterations to achieve a gradient smaller than $\epsilon_g$ for given $\epsilon_g$, and at most $\mathcal O(\max\{ \epsilon_g^{-\frac{3}{2}}, \epsilon_H^{-3} \})$ iterations to reach a second-order stationary point respectively. Notably, the proposed algorithm remains applicable even in cases of the gradient and Hessian of the objective function unknown. Numerical experiments are performed with  gradient and Hessian being approximated by forward finite-differences to illustrate the theoretical results and numerical comparison.\\


\end{abstract}

\maketitle
\section{Introduction}

The main objective of this paper is to develop a Riemannian quasi-Newton method with cubic regularization (R-NMCR, for short) for solving the smooth unconstrained (possibly nonconvex) optimization problems
\begin{equation}\label{eq:mainprob}
\min_{p\in {\cal M} } f(p),
\end{equation}
where ${\cal M}$ represents a given Riemannian manifold, and $f \colon {\cal M} \to \mathbb{R}$ is a sufficiently smooth cost function.

Before proceeding, we will briefly review the literature on the Newton method with cubic regularization (E-NMCR) in Euclidean spaces, that is,  ${\cal M}=\mathbb{R}^n$. It is well-known that Nesterov and Polyak \cite{nesterovpolyak2006} proposed the E-NMCR method to obtain an approximate solution of \eqref{eq:mainprob} starting from every point $x_0\in \mathbb{R}^n$ by solving the subproblem
\begin{equation}\label{sub.pro}
u_{k+1}=\argmin_{u\in \mathbb{R}^n} f(x_k)+ \langle \nabla f(x_k), u\rangle+\frac{1}{2}\langle \nabla^2 f(x_k)u, u\rangle+\frac{L_f}{6}|u|^3.
\end{equation}
 Then, the next point $x_{k+1}$ is defined by $x_{k+1}:=x_k+u_{k+1}$ for all $k\ge 0$. Here, $\nabla^2 f$ is assumed to be $L$-Lipschitz continuous, and $L_f\geq L>0$ is an estimate for $L$.

A fact is that saddle points in nonconvex problems may still pose challenges. Due to the absence of higher-order knowledge, first-order methods can only guarantee convergence to stationary points and lack control over the possibility of getting stuck at saddle points. Alternatively, second-order algorithms typically excel at avoiding saddle points by leveraging curvature information. It is known that standard assumptions allow E-NMCR to escape strict saddle points, as seen in \cite{kholer2017,Tripuraneni2018,Weiwei2013,AbsiBakerGallivan2007,grapiglia2022cubic,boumal2023intromanifolds}. This serves as one of the motivations to continue studying NMCR methods.

It has been shown that E-NMCR produces an iterate $x_k$ with $\|\nabla f(x_k)\|\leq \epsilon,$ for some given $\epsilon>0$, in at most $\mathcal O(\epsilon^{-\frac{3}{2}})$ iterations. Thanks to this optimal complexity result,  Newton's method with cubic regularization was proposed \cite{grapiglia2022cubic,agarwal2021adaptive}. As we can see, E-NMCR solves a cubic model approximating $f$ in each iteration, wherein the full Hessian matrix must be calculated. This may render E-NMCR less competitive a priori or even infeasible if the Hessian is unavailable. To overcome these drawbacks, an adaptive regularization was established for E-NMCR. In these adaptive schemes, subproblem \eqref{sub.pro} is solved inexactly to reduce computational costs, as highlighted in \cite{grapiglia2022cubic,agarwal2021adaptive,birgin2019,cartis2012oracle}. To clarify, in adaptive schemes, the subproblem is addressed as follows:
\begin{equation}\label{adap.sub.pro}
    u_{k+1}=\argmin_{u\in \mathbb{R}^n} f(x_k)+ \langle \nabla f(x_k), u\rangle+\frac{1}{2}\langle H_ku, u\rangle+\frac{\sigma_k}{6}\|u\|^3,
\end{equation}
where $H_k$ satisfies some form of inexact condition, and $\sigma_k>0$ can be chosen in various ways. In \cite{CartisToint2011,CartisToint20112}, it is proposed that 
\begin{equation}\label{1.inex.cond}
\|(H_k-\nabla^2 f(x_k))u_{k+1}\|\leq \eta_1\|u_{k+1}\|^2
\end{equation}
holds for some matrix $H_k$ and $\eta_1\geq 0.$

It is evident that at iteration $k$ of subproblem \eqref{adap.sub.pro}, knowledge of $H_k$ is necessary. However, obtaining $H_k$ itself requires knowledge of $x_{k+1}$ because $H_k$ must satisfy the inexact condition in \eqref{1.inex.cond}. Thus, the implementation of methods involving conditions like \eqref{1.inex.cond} demands additional computational effort. This is most clearly observed in the complexity result derived in \cite{cartis2012oracle}, which is $\mathcal O(m[\epsilon^{-\frac{3}{2}}+|\log(\epsilon)|]),$ where $m$ is the dimension of the domain of the objective function. To enhance this complexity result, \cite{grapiglia2022cubic, WangLan2019} have proposed the following inexactness condition:
\begin{equation}\label{2.inex.cond}
\|H_k-\nabla^2 f(x_k)\|\leq \eta_2\|u_{k}\|,
\end{equation}
with $\eta_2\geq 0,$ which no longer involves the subsequent iteration.

For cubic model \eqref{adap.sub.pro}, an E-NMCR algorithm \cite{grapiglia2022cubic} was recently proposed based on the combination of inexact condition \eqref{2.inex.cond}, approximated Hessian computed by the finite difference method and nonmonotonic line search with the complexity .
$\mathcal O(m\epsilon^{-\frac{3}{2}}),$ where $m$ is the dimension of the domain of the objective function. Furthermore, the E-NMCR with finite difference updates on the Hessian approximation requires at most $\mathcal O(m\max\{\epsilon_g^{-\frac{3}{2}}, \epsilon_H^{-3}\})$ iterations to find an approximate second-order stationary point,  that is, an iterate $x_k$ such that
\begin{equation}\label{1.secordem}
\|\nabla f(x_k)\|\leq \epsilon_g \quad\text{and}\quad \lambda_{\min}(\nabla^2 f(x_k))\geq -\epsilon_H,
\end{equation}
where $\lambda_{\min}(\nabla^2 f(x_k))$ denotes the smallest eigenvalue of $\nabla^2 f(x_k).$

{\bf Related works on manifolds:}  Similarly, in \cite{ZhangZhang2018}, the following cubic subproblem is proposed to analyze optimization problem \eqref{eq:mainprob}
$$
\begin{cases}
    \begin{array}{cc}
        v_k:=\argmin_v \hat{f}_k(0)+\langle \nabla \hat{f}_k(0) , v \rangle+\dfrac{1}{2}\langle \nabla^2\hat{f}_k(0)[v],v\rangle+\dfrac{\sigma}{6}\|v\|^{3} \\
        p_{k+1}:=R(p_k, v_k),
    \end{array}
\end{cases}
$$
where $v\in \T_{p_k}{\cal M},$  $\hat{f}_{k}=f\circ \R_{p_k}\colon\T_{p_k}{\cal M}\to \mathbb{R}$ is the pullback associated with $f$, $\R(\cdot, \cdot)$ is a retraction on ${\cal M},$ and $\sigma>0$ is an estimate for the Lipschitz Hessian constant. To achieve the same complexity as in Euclidan space, some conditions are assumed in \cite{ZhangZhang2018}, and consequently, some constants need to be known, namely:
\begin{equation}\label{eq:ZHACOND}
\begin{cases}
{\cal M} \text{ must be compact}, \quad \|\R(x,p)-x-p\|\leq L_2\|p\|^2, \quad \text{for all} \quad x\in {\cal{M}}, \quad p\in \T_{p}{\cal M},\\
\left|\langle (\nabla^2_{\xi} \hat{f}_x(\eta) - \nabla^2_{\xi} \hat{f}_x(0))[v], v\rangle  \right| \leq L^{\cal R}_{H}\|\eta\|, \,\, \forall \, \eta \in \T_{x} {\cal M},\;\; \|\eta\|\leq {\cal R}, \quad \forall v\in \T_{x} {\cal M}, \quad \|v\|=1,\\
G:=\max_{x\in{\cal M}} \|\nabla f(x)\|_F,\\
\kappa_B:=\max_{x\in{\cal M}}\max_{\xi\in \T_x{\cal M}, \|\xi\|=1} \|\hess f(x)[\xi]\|,\\
{\cal R}=3\kappa_B+3\sqrt{G}.
\end{cases}
\end{equation}
Moreover, to execute the algorithm proposed in \cite{ZhangZhang2018}, it is necessary to choose $\sigma$ such that
\begin{equation}\label{eq:sigmaZhang}
\sigma>\max\left\{ \left(\sqrt{10L_2\kappa_B+\frac{2}{3}L^{\cal R}_{H}+9L_2^2G}+3L_2\sqrt{G}\right)^2, 1\right\},
\end{equation}
where $L_2, \kappa_B, L^{\cal R}_{H}, G$ are the constants defined in \eqref{eq:ZHACOND}.  Thus, it is not difficult to see that the algorithm proposed in \cite{ZhangZhang2018} can become impractical. As in the Euclidean context, it was proved that the R-NMCR finds an approximate second-order stationary point within $\mathcal O(\epsilon^{-\frac{3}{2}})$ iterations \cite{ZhangZhang2018}.

In \cite{agarwal2021adaptive}, a more general algorithm was proposed to approximately solve problem  \eqref{eq:mainprob}. Specifically, (i) the two main results in \cite{agarwal2021adaptive} can be applied to every complete Riemannian manifold such that the exponential and retraction maps can be used, (ii) the subproblem in \cite{agarwal2021adaptive} is the same as the one studied in \cite{ZhangZhang2018}, but $\sigma$ is adaptively chosen and does not depend on any constant, as in \eqref{eq:sigmaZhang}, (iii) the complexity order  $\mathcal O(\epsilon^{-\frac{3}{2}})$ is guaranteed when using both the exponential map and a general retraction.

Consider subsample and cubic regularization techiques to approximately solve the problem \cite{DengMu2023}
\begin{equation*}
\min_{p\in {\cal M} } f(p):=\frac{1}{n}\sum_{i=1}^n f_i(p),
\end{equation*}
where $f_i \colon {\cal M} \to \mathbb{R}$ is a sufficiently smooth cost or loss function for each $i\in\{1,\ldots, n\}$. In this case, the suproblem takes the following structure:
\begin{equation}\label{eq:BOUDEN}
\begin{cases}
    \begin{array}{cc}
        v_k:=\argmin_v \hat{f}_k(0)+\langle \mathcal{G}_k , v \rangle+\dfrac{1}{2}\langle \mathcal{H}_k[v],v\rangle+\dfrac{\sigma_k}{3}\|v\|^{3} \\
        p_{k+1}:=R(p_k, v_k),
    \end{array}
\end{cases}
\end{equation}
with $v\in \T_{p_k}{\cal M}$ and $\mathcal{G}_k$ and $\mathcal{H}_k[v]$ being, respectively, the approximated Riemannian gradient and Hessian calculated using the subsampling technique, i.e.,
\begin{equation}\label{eq:sumGH}
\mathcal{G}_k=\frac{1}{|S_g|}\sum_{i\in S_g} \grad f_i(p_k) \quad \text{and}\quad\mathcal{H}_k[v]:=\frac{1}{|S_H|}\sum_{i\in S_H} \hess f_i(p_k)[v],
\end{equation}
where $S_g, S_H\in \{1,\ldots, n\}$ are the sets of the subsampled indices used for estimating the Riemannian gradient and Hessian, respectively. It is straighforward to see that if $n=1$, then $i=1,$ $\mathcal{G}_k=\grad f(p_k),$ $\mathcal{H}_k[v]=\hess f(p_k)[v],$ and hence subproblem \eqref{eq:BOUDEN} is the same as in \cite{agarwal2021adaptive,ZhangZhang2018}. To prove the main results in \cite{DengMu2023}, some strong assumptions are made, for instance: (i) the knowledge of the Lipschitz Hessian constant, (ii) the knowledge of a constant that bounds the inexact Hessian defined in \eqref{eq:sumGH}, which consequently implies that $\sigma_k$ depends on these constants. Under these conditions, the best second-order complexity achievement obtained in \cite{DengMu2023} is $\mathcal O(\max\{\epsilon_g^{-2}, \epsilon_H^{-3}\}),$ where $\epsilon_g$ and $\epsilon_H$ are as in \eqref{1.secordem}. To improve this complexity results,  an additional condition was assumed on the solution of the subproblem, namely, $v_k$ must satisfy the following system:
\begin{equation*}
   \begin{cases}
    \begin{array}{cc} \langle G_k, v_k\rangle+\langle H_k[v_k],v_k\rangle+\sigma_k\|v_k\|^3=0\\
    \langle H_k[v_k],v_k\rangle+\sigma_k\|v_k\|^3\geq 0.
    \end{array}
\end{cases}
\end{equation*}
Thus, the new  second-order complexity result in \cite{DengMu2023} is $\mathcal O(\max\{\epsilon_g^{-\frac{3}{2}}, \epsilon_H^{-3}\}).$

{\bf Our contributions:} A new R-NMCR is proposed to approximately solve \eqref{eq:mainprob} that is entirely adaptive which means that Lipschitz gradient or Hessian constants are not necessary known in advance, neither the regularization parameter of the cubic models and the accuracy of the Hessian approximations
jointly adjusted by using a nonmonotone line search criterion. The main results obtained here are applicable to all complete Riemannian manifold where the exponential and retraction maps can be utilized. Our subproblem is also inexactly solved in the sense of approximated Riemannian gradient and Hessian. Moreover, under standard assumptions, the proposed algorithm requires at most $\mathcal O(\max\{\epsilon_g^{-\frac{3}{2}}, \epsilon_H^{-3}\})$ iterations to achieve a second-order stationary point. This means obtaining a point $p_k$ that satisfies a condition similar to \eqref{1.secordem} but in the Riemannian context. Finally, the new algorithm can be applied when the gradient and Hessian approximations are computed by using forward finite-differences.

The subsequent sections of this paper are structured as follows. Section 2 provides a concise review of the preliminaries. In Section 3, the primary derivative-free algorithm proposed are introduced. The worst-case complexity of the proposed algorithm is analyzed in Section 4. Section 5  detailed insights into the computation of the approximated Riemannian gradient and Hessian are given. In Section 6, the results of numerical tests conducted to showcase the effectiveness of the proposed algorithm are displayed. Finally a summary and concluding remarks are given in the last section.



\section{Preliminary }

In this section, we review notations, definitions, and basic properties related to Riemannian manifolds, which are utilized throughout the paper. These concepts can be found in introductory books on Riemannian geometry and optimization on manifolds, such as \cite{DoCa92, lee2006riemannian, Lee:2003:1, Tu:2011:1} and \cite{boumal2023intromanifolds,absil2008optimization}.

Suppose that ${\cal M}$ is a $n$-dimensional connected, smooth manifold. The tangent space at $p\in {\cal M}$ is a $n$-dimensional vector space denoted by $\T_p{\cal M}$ whose origin is $0_p$. The disjoint union of all tangent spaces $ \T{\cal M} \coloneqq \cup_{p\in {\cal M}} ( \{p\} \times \T_p{\cal M} ) $ is called the {\it tangent bundle} of ${\cal M}$. The Riemannian metric at $p\in {\cal M}$ is denoted by $ {\langle}  \cdot  ,   \cdot  {\rangle}_p\colon \T_p{\cal M} \times \T_p{\cal M} \to \mathbb{R}$ and $\|\cdot \|_p$ for the associated norm in $ \T_p{\cal M}$.
For simplicity we shall omit all these indices when no ambiguity arises.
Assume that ${\cal M}$ is equipped with a Riemannian metric, that is a {\it Riemannian manifold}.

 A {\it vector field} $V$ on ${\cal M}$ is a correspondence associated to each point $p\in {\cal M}$ a vector $V(p)\in \T_p{\cal M} $. Let us denote the  smooth vector fields on ${\cal M} $ by ${\cal X}({\cal M} )$ and $\bar{\nabla} \colon {\cal X}({\cal M} )  \times {\cal X}({\cal M} ) \to {\cal X}({\cal M} )$ for the Levi-Civita connection associated to $\mathcal{M}$.  The tangent vector of a smooth curve $\gamma\colon I  \to {\cal M}$ defined on some open interval $ I\subseteq \mathbb{R}$  is denoted by $ \dot{\gamma}(t)$. A vector field $V$ along a smooth curve $\gamma$  is said to be {\it parallel} if and only if $\bar{\nabla}_{ \dot{\gamma} } V=0$. The $\gamma$ is a {\it geodesic} when $ \dot{\gamma}$ is self-parallel. When the geodesic equation $\bar{\nabla}_{\dot{\gamma}}  \dot{\gamma}=0$ is a second order nonlinear ordinary differential equation, the geodesic $\gamma=\gamma _{v}( \cdot ,p)$ is determined by its position $p$ and velocity $v$ at $p$.  A Riemannian manifold is {\it complete} if the geodesics are defined for all values of $t\in \mathbb{R}$. Owing to  the completeness of the Riemannian manifold $\mathcal{M}$, the {\it exponential map} $\exp_{p}\colon \T_{p}  \mathcal{M} \to \mathcal{M} $ is  given by $\exp_{p}v\,=\, \gamma _{v}(1,p)$, for each $p\in \mathcal{M}$. Next, a detailed definition is provided for a map that generalizes the exponential map and plays a crucial role in the approach presented in this paper.

 \begin{definition}[{\cite[Definition 4.1.1 and Sect. 5]{absil2008optimization}}]
	A retraction on  ${\cal M}$ is a smooth map
	$$
	\R \colon \T{\cal M} \to  {\cal M} \colon (p,v) \mapsto \R_{p}(v)
	$$
	such that each curve $c(t)= \R_p(tv)$ satisfies $c(0)=p$ and $c'(0)=v$. Retractions that additionally satisfy $c''(0)=0$ are termed second-order retractions.
\end{definition}
The set of smooth scalar functions on $\cM$ is denoted by $\mathcal{F}(\cal{M})$. The {\it differential} of a function $f\in \mathcal{F}(\cal{M})$ at $p$ is the linear map ${\cal D}f(p)\colon  \T_p{\cal M} \to \mathbb{R}$ which assigns to each $v\in \T_p{\cal M}$ the value
$$
{\cal D}f(p)[v] = \dot{\gamma}(t_0)[f] = \frac{\de}{\dt}(f\circ \gamma)\Bigl|_{t=t_0},
$$
for every smooth curve $\gamma\colon I  \to {\cal M}$ satisfying $\gamma(t_0)=p$ and $\dot{\gamma}(t_0)=v$. The {\it Riemannian  gradient} at $p$ of $f$, $\grad f(p)$, is defined by the unique tangent vector at $p$ such that ${\langle}  \grad f(p) , v  {\rangle}_p = {\cal D}f(p)[v]$ for all $v\in \T_p{\cal M}.$ The {\it Riemannian Hessian} of $f \in \mathcal{F}(\cal{M}) $ at $p \in \cal{M}$ is a linear operator $\hess f(p): \T_{p}{\cal{M}} \to {\T_{p}}{\cal{M}}$ defined as
	$
	\hess f(p)[u] = \bar{\nabla}_{u} \grad f.
	$
	For real functions on vector spaces (such as $\T_p\cM$), we let $\nabla$ and $\nabla^2$ denote the usual
gradient and Hessian.
	The norm of a linear map $ A \colon  \T_p{\cal M} \to  \T_p{\cal M}$ is defined  by $\|A\|_{\op} \coloneqq \sup\{  \| A v \| \colon  v\in \T_p{\cal M}, \, \|v\|=1  \}$.

 For each $t_0,t \in I $, $t_0 <t$,  the connection $\bar{\nabla}$ induces an isometry  $\p_{\gamma,t_0,t} \colon \T_{\gamma(t_0)} {\mathcal{M}} \to \T_{\gamma(t)} {\mathcal{M}}$  relative to Riemannian metric on ${\cal M}$ given by $ \p_{\gamma,t_0,t}\, v = V(\gamma(t))$, where $V$ is the unique vector field on $\gamma$ such that $ \bar{\nabla}_{\dot{\gamma}(t)}V(\gamma(t)) = 0$ and $V(\gamma(t_0))=v$.

  The isometry $\p_{\gamma,t_0,t}$  is called {\it parallel transport} along  of  $\gamma$ joining  $\gamma(t_0)$ to $\gamma(t)$.
For simplicity $\p_v \colon \T_p{\cal M} \to \T_{\exp_pv}{\cal M} $ denotes parallel transport along the geodesic $\gamma(t)= \exp_p tv$ from $t_0=0$ to $t=1$.

\begin{definition}[{\cite[Definition 2]{agarwal2021adaptive}}]
A function $f\colon \cM \to \mR$ has an $L$-Lipschitz continuous Hessian if it is twice differentiable and if
	\begin{equation*}
		\left\| \p_{v}^{-1}\circ \hess f(\exp_pv) \circ \p_{v} - \hess f(p)\right\|_{\op} \leq L \|v\| \mbox{ \,\, for all \,\,} (p,v)\in \T\cM.
	\end{equation*}
\end{definition}
  The following Lemma provides classic inequalities that will be explored extensively throughout this paper.
 \begin{lemma}[{\cite[Proposition 2]{agarwal2021adaptive}}]\label{lem:ineq.hess.lips}
  Let $f\colon \cM  \to \mR$ be twice differentiable on a complete Riemannian manifold $\cM$.  If $f\colon \cM \to \mR$ has an $L$-Lipschitz continuous Hessian then
\begin{equation*}\label{}
	\left| f(\exp_pv) - f(p) - \langle\grad f(p),v\rangle - \dfrac{1}{2}\langle\hess f(p)[v],v\rangle \right| \leq \dfrac{L}{6}\|v\|^{3},
\end{equation*}
and
\begin{equation*} \label{}
	\left\| \p_{v}^{-1} \grad f(\exp_pv) - \grad f(p) - \hess f(p)[v]  \right\|\leq\dfrac{L}{2}\|v\|^{2}.
\end{equation*}
\end{lemma}

\begin{lemma}[{\cite[Lemma 4]{grapiglia2022cubic}}]\label{lem:ineq.num.aux}
Given two real constants $a,b>0$ and a set $\left\{z_{k} \colon k= 1, \ldots, \bar{N}\right\}$ of nonnegative real numbers, with $\bar{N}\geq 2$ natural, let
$
\bar{k} \coloneqq \argmin_{ k \in \{1,\ldots,\bar{N}-1 \} } ( ( z_{k} )^{a} + ( z_{k+1} )^{a} ).
$
If
$
\sum_{k=1}^{\bar{N}} (z_{k})^{a}\leq b
$
then the inequality
$
\max \{ z_{\bar{k}},z_{\bar{k}+1} \} \leq \left( 2b /(\bar{N}-1) \right)^{1/a}
$
holds.
\end{lemma}
 {\it In this paper, all manifolds $\cM$ are assumed to be  Riemanian, connected, finite dimensional, and complete.}



\section{The Riemannian quasi-Newton Method}

In this section, our aim is to introduce a comprehensive Riemannian quasi-Newton method with cubic regularization and demonstrate its convergence. The algorithm presented below is applicable when both the gradient and Hessian are known. Alternatively, it can be employed with any approximation of the gradient and Hessian that satisfies the assumptions outlined in Step 1.1.

\begin{mdframed}
	\textbf{Algorithm 1: General R-NMCR}
	\\[0.2cm]
	\textbf{Step 0.} Choose a retraction $\R$,  a point $(p_{0}, v_{0}) \in \T {\cal M}$ ($v_{0}\neq 0$) and constants $\sigma_{1} > 0$ and $\theta \geq 0$. Set $k = 1$.
	\\
	{\bf Step 1.}  Find the smallest integer $\alpha \geq 0$ such that $2^{\alpha-1}\sigma_{k} \geq \sigma_{1}$.\\
	 {\bf Step 1.1.}  Choose a vector $g_{k,\alpha}\in  \T_{p_k} {\cal M}$ and an operator $\B_{k,\alpha} \colon  \T_{p_k} {\cal M} \to  \T_{p_k} {\cal M}$ that satisfy
		\begin{equation}\label{eq:grahess.alg}
		\| \grad f(p_k) - g_{k,\alpha} \| \leq \frac{ \kappa_{g} }{2^{\alpha-1} } \| v_{k-1} \|^2,\qquad\qquad
		\| \hess f(p_k) - \B_{k,\alpha} \|_{\op} \leq \frac{\kappa_{\B} }{2^{\alpha-1} } \| v_{k-1} \|,
		\end{equation}
		for fixed constants $ \kappa_g, \kappa_{\B} \geq 0$ that are independent of $k$ and $\alpha$.\\
	{\bf Step 1.2.} Consider the cubic model  $m_{k,\alpha}$ on $\T_{p_k}\mathcal{M}$ given by
	\begin{equation}\label{eq:model.f}
	m_{k, \alpha}(v) = f(p_k)+\langle g_{k,\alpha} , v \rangle+\dfrac{1}{2}\langle {\B}_{k,\alpha}[v],v\rangle+\dfrac{2^{\alpha}\sigma_k}{6}\|v\|^{3},
        \end{equation}
and compute an approximate minimizer $v_{k,\alpha}$ of $m_{k,\alpha}$ over $\T_{p_k}\mathcal{M}$ that satisfies
	\begin{equation}\label{eq:consubpalg}
	m_{k,\alpha}( v_{k,\alpha} )\leq f(p_{k})\quad\text{and}\quad \|\nabla m_{k,\alpha}(v_{k,\alpha} )\|\leq\theta \| v_{k,\alpha} \|^2.
	\end{equation}
	Optionally, if second-order criticality is targeted, $v_{k,\alpha}$ must also satisfy
	\begin{equation}\label{eq:icrehe3}
          \lambda_{\min}( \B_{k,\alpha} ) \geq -   2^{\alpha-1}\sigma_{k}  \| v_{k,\alpha} \|
           -  \theta \| v_{k-1} \|.
       \end{equation}
	{\bf Step 1.3.} If
	\begin{equation}\label{eq:cond.mono.alg}
		f \left( \R_{p_k}(v_{k,\alpha} ) \right)
			 \leq  f(p_k) + \frac{ \sigma_k }{24} \| v_{k-1} \|^3
			- \frac{  2^{\alpha } \sigma_k }{24}\| v_{k,\alpha} \|^{3}
	\end{equation}
	hold, set $\alpha_{k}=\alpha$, $v_k = v_{k,\alpha_k}$ and go to Step 2. Otherwise, set $\alpha \coloneqq \alpha + 1$ and return to Step 1.1.
	\\
	{\bf Step 2.} Set $p_{k+1}=\R_{p_k}( v_k )$, $\sigma_{k+1}=2^{\alpha_{k}-1}\sigma_{k}$, $k:=k+1$, and return to Step 1.
\end{mdframed}

\begin{remark}
 Note that Algorithm 1 works well when $g_{k,\alpha} = \grad f(p_k)$ and $\B_{k,\alpha} = \hess f(p_k)$ known for all $\alpha \geq 0$ because the inequalities in \eqref{eq:grahess.alg} are satisfied naturally. Furthermore, closed-form expressions of the approximations $g_{k,\alpha}$ and $\B_{k,\alpha}$ not only satisfy \eqref{eq:grahess.alg}, but also eliminate evaluations of the gradient and Hessian. Hence, a Derivative-Free R-NMCR algorithm can be developed from the general R-NMCR algorithm with these closed-form expressions for $g_{k,\alpha}$ and $\B_{k,\alpha}$.

\end{remark}

\begin{remark}

 For implementation of Algorithm 1, instead of constants $\kappa_g$ and $\kappa_{\B}$ known, it is possible to choose $g_{k,\alpha}$ and $\B_{k, \alpha}$ satisfying \eqref{eq:grahess.alg} for $\kappa_g$ and $\kappa_{\B}$ unknown. This flexibility is particularly crucial when $\kappa_g$ and $\kappa_{\B}$ depend on the Lipschitz constant of $\hess f$. Some examples of approximations of $g_{k,\alpha}$ and $\B_{k,\alpha}$ satisfying \eqref{eq:grahess.alg}, with $\kappa_g$ and $\kappa_{\B}$ dependent on the Lipschitz constant of $\hess f$, are given in Section~\ref{sec:FDAGH}.
\end{remark}


For theoretical analysis, some basic assumptions for the cost function $f$ of problem \eqref{eq:mainprob} are given as follows. The first one is very common and says that the cost function $f$ is lower bounded.

\begin{assumption}\label{Ass1}
There exists $f_{\low}\in\mathbb{R}$ such that $f(p)\geq f_{\low}$ for all $p\in  {\cal M}$.
\end{assumption}

Before presenting the second assumption, we shall introduce a notation that will be used throughout this paper. For a given cost function $f$ and a specified retraction $\R$, at each iterate $p_k$ of Algorithm 1, we will consider the following notation:
\begin{equation}\label{notation}
\hat {f}_k  \coloneqq f \circ \R_{p_k} \colon \T_{p_k} \cM \to \mathbb{R}.
\end{equation}
The function $\hat {f}_k$ is often called the pullback of the cost function $f$ to the tangent space $\T_{p_k} \cM$.


\begin{assumption}\label{Ass2}
The function $f$ is twice continuously differentiable and there exists a constant $L$ such that, at each iteration $k$, the inequality
\begin{equation}\label{eq:Ass2}
\left| \hat{f}_k(v) -  f(p_k) - \langle\grad f(p_k),v\rangle - \dfrac{1}{2}\langle\hess f(p_k)[v],v\rangle \right| \leq \frac{L}{6}\|v\|^{3} \,\, \mbox{ holds for all } \, v\in \T_{p_k} {\cal M}.
\end{equation}
\end{assumption}

\begin{remark}\label{rem:Ass2}
It follows from the definition of $\hat{f}_k$ in \eqref{notation} that
\begin{equation}\label{eq:rel.betw.f.hf}
f(p_k) = \hat{f}_k(0) \quad \text{and} \quad \grad f(p_k) = \nabla \hat{f}_k(0) .
\end{equation}
Moreover, if the retraction $\R$ employed in the definition of $\hat{f}_k$ is a second-order retraction, there is
\begin{equation}\label{eq:rel.betw.f.hf-hess}
\hess f(p_k) =\nabla^2  \hat{f}_k(0),
\end{equation}
as demonstrated in {\cite[Proposition 5.45]{absil2008optimization}}. Therefore, whenever $\R$ is a second-order retraction (e.g., $\R=\exp$), \eqref{eq:Ass2} can be reformulated as
\begin{equation}\label{eq:Ass2.r}
\left| \hat{f}_k(v) -  \hat{f}_k(0) - \langle \nabla  \hat{f}_k(0) ,v\rangle - \dfrac{1}{2}\langle \nabla^2  \hat{f}_k(0) [v],v\rangle \right| \leq \frac{L}{6}\|v\|^{3}  \,\, \mbox{ for all } \, v\in \T_{p_k} {\cal M}.
\end{equation}
On the other hand, as $\hat{f}_k$ is defined on the vector space $\T_{p_k} {\cal M}$, \eqref{eq:Ass2.r} holds for $L$-Lipschitz $ \nabla^2 \hat{f}_k$. Overall, it can be asserted that a sufficient condition for Assumption \ref{Ass2} to be satisfied is that $R$ and $ \nabla^2 \hat{f}_k$ are a second-order retraction and $L$-Lipschitz respectively.
\end{remark}



Under a reasonable assumption (Assumption 2, to be more specific), the next result guarantees that Algorithm 1 is well-defined, that is, the existence of $\alpha \in [0, + \infty)$ satisfying condition \eqref{eq:cond.mono.alg}.
\begin{theorem}
	\label{teo:good.def}
	Suppose that  Assumption~\ref{Ass2} holds. For every iteration $k$, if $\alpha\geq0$ satisfies
		\begin{equation}\label{eq:bmvmsig}
             2^{\alpha-1} \sigma_k \geq 12( 2\kappa_{g} + \kappa_{\B} ) + L
	\end{equation}
        then
	\begin{equation}\label{eq:lem2inc}
		\hat{f}_k(v_{k,\alpha} )
			 \leq  f(p_k) + \frac{ \sigma_k }{24} \| v_{k-1} \|^3
			- \frac{  2^{\alpha } \sigma_k }{24}\| v_{k,\alpha} \|^{3}.
	\end{equation}
\end{theorem}
\begin{proof}	
Take an arbitrary iteration $k$ and a constant $\alpha\geq 0$ satisfying \eqref{eq:bmvmsig}.
 It follows from Assumption~\ref{Ass2} with $v=v_{k,\alpha}$,  the definition of $m_{k,\alpha}$ (given in \eqref{eq:model.f}), the first inequality of \eqref{eq:consubpalg}, and \eqref{eq:grahess.alg} that
		\begin{align*}
		\MoveEqLeft
		\hat{f}_k(v_{k,\alpha} ) \\
		& \leq   f(p_k)+ \langle\grad f(p_k),v_{k,\alpha} \rangle +\frac{1}{2}\langle\hess f(p_k)[v_{k,\alpha}],v_{k,\alpha}\rangle+\frac{L}{6}\| v_{k,\alpha} \|^{3}\nonumber \\
		& =  m_{k,\alpha}(v_{k,\alpha}) + \langle\grad f(p_k) - g_{k,\alpha},v_{k,\alpha} \rangle +\frac{1}{2} \left\langle  (\hess f(p_k)-{\B_{k,\alpha}})[v_{k,\alpha}],v_{k,\alpha} \right\rangle +\frac{L-2^{\alpha} \sigma_k}{6}\|v_{k,\alpha}\|^{3}\nonumber \\
			& \leq  f(p_k)+ \| \grad f(p_k) - g_{k,\alpha} \| \| v_{k,\alpha} \| + \frac{1}{2} \left\| \hess f(p_k)-{\B_{k,\alpha}} \right\|_{\op} \| v_{k,\alpha} \|^2 +\frac{L-2^{\alpha} \sigma_k}{6}\| v_{k,\alpha} \|^{3}\nonumber \\
			& \leq  f(p_k) + \frac{ \kappa_{g} }{ 2^{\alpha-1} }  \| v_{k-1} \|^2\| v_{k,\alpha} \| + \frac{ \kappa_{\B} }{ 2^{\alpha} }  \| v_{k-1} \| \| v_{k,\alpha} \|^2
			+\frac{L-2^{\alpha} \sigma_k}{6}\| v_{k,\alpha} \|^{3}. \nonumber
	\end{align*}
	Since
	$
	\| v_{k-1} \|^2\| v_{k,\alpha} \| \leq  \| v_{k-1} \|^3 + \| v_{k,\alpha} \|^3
	$
	and
	$
	\| v_{k-1} \|\| v_{k,\alpha} \|^2 \leq  \| v_{k-1} \|^3 + \| v_{k,\alpha} \|^3
	$, it follows that
	\begin{equation*}
		\hat{f}_k(v_{k,\alpha} )
			 \leq  f(p_k) + \frac{ 2\kappa_{g} +  \kappa_{\B} }{2^{\alpha}} \| v_{k-1} \|^3
			 +\frac{ 2 \kappa_{g} +   \kappa_{\B} }{ 2^{ \alpha} }\| v_{k,\alpha} \|^{3}
			+\frac{  L - 2^{\alpha} \sigma_k}{6}\| v_{k,\alpha} \|^{3}.
	\end{equation*}
	By means of \eqref{eq:bmvmsig}, we can ensure that $  (2\kappa_{g} + \kappa_{\B} )/ 2^{\alpha}  \leq \sigma_{k}/ 24 $ holds for every $\alpha \geq 0$. Thus, the previous inequality leads to the following
	\begin{equation*}
		\hat{f}_k(v_{k,\alpha} )
			 \leq  f(p_k) + \frac{ \sigma_k }{24} \| v_{k-1} \|^3
			+\frac{ \sigma_k + 4L  - 2^{\alpha+2} \sigma_k }{24}\| v_{k,\alpha} \|^{3}.
	\end{equation*}
	By using \eqref{eq:bmvmsig} again, one can easily conclude that
	$$
	 \sigma_k + 4L  - 2^{\alpha+2} \sigma_k  =  ( \sigma_k - 2^{\alpha } \sigma_k ) + (4L  - 2^{\alpha+1} \sigma_k)  - 2^{\alpha} \sigma_k
	 \leq   - 2^{\alpha} \sigma_k
	$$
	for all $\alpha \geq 0$. Finally, the previous inequality implies that \eqref{eq:lem2inc} is true.
	 \end{proof}
The next result shows that the sequence $\{\sigma_k\}$ of regularization parameters is bounded, and further provides lower and upper bounds for this sequence.

\begin{corollary}\label{cor:gooddef}
	Under  Assumption~\ref{Ass2}, the sequence of regularization parameters $\left\{\sigma_{k}\right\}$ in Algorithm 1 satisfies
	\begin{equation}\label{eq:bound.sigmak}
		\sigma_{1}\leq\sigma_{k}\leq  24 ( 2 \kappa_{g} +  \kappa_{\B} ) + 2L  + \sigma_{1}  \coloneqq \sigma_{\max}, \qquad k=1,2,\ldots
	\end{equation}
\end{corollary}

\begin{proof}
	Clearly, \eqref{eq:bound.sigmak} is true for $k =1$, and thus our induction base holds. Suppose that \eqref{eq:bound.sigmak} holds for some $k \geq 1$. If $\alpha_{k}=0$, then by Step 1 and the induction hypothesis, we have
	\begin{equation*}
		\sigma_{1}\leq\sigma_{k +1}= 2^{-1}\sigma_{k}\leq\sigma_{k}\leq \sigma_{\max},
	\end{equation*}
	that is, \eqref{eq:bound.sigmak} holds for $k+1$. On the other hand, if $\alpha_{k}\geq 1,$ then there is
	\begin{equation}\label{eq:ineq.ins.cor.gd}
		2^{\alpha_{k}-1}\sigma_{k} \leq \sigma_{\max}
	\end{equation}
	Indeed, by assuming that \eqref{eq:ineq.ins.cor.gd} is not true, it follows that
	$$
		2^{ \alpha_{k} - 2 }\sigma_{k} >  2^{-1} \sigma_{\max}> 12\left( 2\kappa_{g} + \kappa_{\B} \right) + L .
	$$
	In this case, by Theorem \ref{teo:good.def}, inequality \eqref{eq:cond.mono.alg} would have been satisfied for $\alpha=\alpha_{k}-1$, contradicting the minimality of $\alpha_{k}$. Thus, \eqref{eq:ineq.ins.cor.gd} is true. Consequently, it follows from Step 1 and \eqref{eq:ineq.ins.cor.gd} that
	\begin{equation*}
		\sigma_{1}\leq \sigma_{k+1}= 2^{\alpha_{k}-1}\sigma_{k} \leq \sigma_{\max},
	\end{equation*}
	that is, \eqref{eq:bound.sigmak} also holds for $k+1$ in this case. This completes the induction argument.
	
\end{proof}

\section{Worst-Case Iteration Complexity Analysis}

In this section, we provide first- and second-order analysis of Algorithm 1 in different subsections. The next result will support both analyses.

\begin{lemma}\label{lem:aux.conv.expc}
Let $N\geq 3$ be a natural number and define
\begin{equation}\label{eq:defkbar}
\bar{k}\coloneqq \argmin_{k \in \{ 1,\ldots,N - 2 \} }\left\{ \| v_{k} \|^3 + \|v_{k+1}\|^3 \right\}.
\end{equation}
If  Assumptions~\ref{Ass1} and \ref{Ass2} hold then
\begin{equation}\label{lem:aux.con.exp.ld-r}
	 \max\left\{ \| v_{ \bar{k}} \| , \| v_{\bar{k}+1}\| \right\} \leq
	  \left[\frac{48(f(p_{1})-f_{low})}{\sigma_{1}}+2 \|v_0\|^3 \right]^{\frac{1}{3}}\dfrac{1}{(N-2)^{\frac{1}{3}}}.
	\end{equation}
\end{lemma}
\begin{proof}
	Consider \eqref{eq:cond.mono.alg} with $\alpha=\alpha_k$. Since $v_{k,\alpha_k}=v_k$,
       $p_{k+1}=\R_{p_k}(v_k)$, $2 \sigma_{k+1}=2^{\alpha_{k}}\sigma_{k}$ and, by Corollary~\ref{cor:gooddef}, $\sigma_{k}\geq\sigma_{1}$ for all $k\geq 1$, we can conclude that
	\begin{align*}
		       f(p_{k})-f(p_{k+1})  & \geq \frac{2\sigma_{k+1}}{24} \| v_{k} \|^3 - \frac{\sigma_{k}}{24} \| v_{k-1} \|^3   \\
	& =   \frac{\sigma_{k+1}}{24} \| v_{k} \|^3  + \frac{ 1 }{24} \left( \sigma_{k+1} \| v_{k} \|^3 - \sigma_{k} \| v_{k-1} \|^3 \right) \\
	& \geq  \frac{\sigma_{1}}{24} \| v_{k} \|^3  + \frac{ 1 }{24} \left( \sigma_{k+1} \| v_{k} \|^3 - \sigma_{k} \| v_{k-1} \|^3 \right),
	\end{align*}
         for all $k=1,\ldots, N-1$. Summing over $k=1, \ldots, N-1$ and using Assumption~\ref{Ass1}, we get
	\begin{align*}
		f(p_{1})-f_{low}&\geq   \sum_{k=1}^{N-1} f(p_{k})-f(p_{k+1})\\
		&\geq \frac{\sigma_{1}}{24}  \sum_{k=1}^{N-1} \| v_{k} \|^3  + \dfrac{ 1 }{24} \sum_{k=1}^{N-1} \left( \sigma_{k+1} \| v_{k} \|^3 - \sigma_{k} \| v_{k-1} \|^3 \right) \\
	        &\geq  \frac{\sigma_{1}}{24} \sum_{k=1}^{N-1}  \| v_{k} \|^3   - \frac{ \sigma_{1} }{24}  \| v_{0} \|^3,
	\end{align*}
	which is equivalent to
	\begin{equation}\label{eq:sumileexc-r}
		\sum_{k=1}^{N-1} \|v_k\|^3 \leq \frac{24(f(p_{1})-f_{low})}{\sigma_{1}}+ \|v_{0} \|^3.
	\end{equation}
	Thus, \eqref{lem:aux.con.exp.ld-r} follows from \eqref{eq:sumileexc-r} and Lemma \ref{lem:ineq.num.aux} with $z_{k}= \|v_k\|$, $\bar{N}=N-1$ and $a=3$.
\end{proof}


\subsection{First-order analysis}

For our first-order analysis, we need to impose the following additional assumption.

\begin{assumption}\label{Ass3}
Let $f$ be a twice continuously differentiable function, and let ${\cal G} \colon \T{\cal M} \to \T{\cal M}$ be a mapping whose image of ${\cal G}(p, \cdot ) \colon \T_p{\cal M} \to \T{\cal M}$ belongs to $\T_p{\cal M}$ for all $p\in {\cal M}$. Suppose that there exists a constant $L'$ such that, at each iteration $k$, the inequality
\begin{equation}\label{eq:Ass3-lc}
\left\| {\cal G}(p_k,v) - \grad f(p_k) - \hess f(p_k)[v]  \right\| \leq\dfrac{L'}{2}\|v\|^{2} \,\, \mbox{ holds for all } \, v\in \T_{p_k} {\cal M}.
\end{equation}
\end{assumption}

\begin{remark}\label{rem:Ass3}
Due to Lemma~\ref{lem:ineq.hess.lips}, we can assert  that if $\hess f$ is $L'$-Lipshitz then Assumption~\ref{Ass3} holds for ${\cal G} \colon \T{\cal M} \to \T{\cal M}$ defined by ${\cal G}( p ,v) = \p_{v}^{-1} \grad f(\exp_{ p }v)$ for all $(p,v)\in \T\cM$. Furthermore, assume that $\R$ is an arbitrary second-order retraction (not necessarily equal to the exponential map). Hence, it follows from  \eqref{eq:rel.betw.f.hf} and \eqref{eq:rel.betw.f.hf-hess} that \eqref{eq:Ass3-lc} can be rewritten as
\begin{equation*}
\left\| {\cal G}(p_k,v) - \nabla \hat{f}_k(0) - \nabla^2 \hat{f}_k(0)[v]  \right\| \leq\dfrac{L'}{2}\|v\|^{2} \,\, \mbox{ for all } \, v\in \T_{p_k} {\cal M},
\end{equation*}
which ensures that if $\nabla^2 \hat{f}_k$ is $L'$-Lipschitz then Assumption~\ref{Ass3} is satisfied for ${\cal G} \colon \T{\cal M} \to \T{\cal M}$ defined by ${\cal G}( p ,v) = \nabla ( f \circ \R_p)(v)$ for all $(p,v)\in \T\cM$.
\end{remark}

The following lemma provides an important inequality for the next two theorems.

\begin{lemma}\label{lem:bound.calG}
	Suppose that $f$ and ${\cal G}$ satisfy Assumption~\ref{Ass3}. Under Assumption \ref{Ass2}, for every iteration $k$ one has
	\begin{equation}\label{eq:tegdgbgt}
		\| {\cal G}(p_k,v_{k}) \|\leq \tau \max\left\{ \| v_{k-1} \|  , \| v_{k} \| \right\}^{2},
		\qquad \tau \coloneqq \frac{L' + 2(\theta+\sigma_{\max} ) + 4( \kappa_g + \kappa_{\B}) }{2},
	\end{equation}
	where $\sigma_{\max}$ is defined in \eqref{eq:bound.sigmak}.
\end{lemma}
\begin{proof}
It follows from \eqref{eq:model.f} with $\alpha=\alpha_k$ and  $\sigma_{k+1} = 2^{\alpha_k-1} \sigma_k $  that
	\begin{align*}
		\nabla m_{k,\alpha_k}(v_{k} ) & =  g_{k,\alpha_k} +\B_{k,\alpha_k} [v_{k} ]+2^{\alpha_k-1} \sigma_k \|v_{k} \|v_{k}  \\
		 & =  {\cal G}(p_k,v_{k})
		  - \left(  {\cal G}(p_k,v_{k})  - \grad f(p_{k} ) - \hess f(p_k)[v_{k} ] \right)\\
		 & \quad + \left( g_{k,\alpha_k} - \grad f(p_{k}) \right)	+ \left(  \B_{k,\alpha_k} - \hess f(p_k)  \right) [v_{k} ]+  \sigma_{k+1}    \|v_{k} \|v_{k}.
	\end{align*}
	Taking norms on both sides and also using the second inequality of \eqref{eq:consubpalg} with $\alpha=\alpha_k$, we find by triangle inequality that
	\begin{align*}		
	\theta\| v_{k}  \|^2
		 \geq  \left\| \nabla m_{ k , \alpha_k }(v_{k} ) \right\|
		 &\geq  \left\| {\cal G}(p_k,v_{k}) \right\|
		 - \left\| {\cal G}(p_k,v_{k}) -\grad f(p_k) - \hess f(p_k)[v_{k} ]  \right\|  \\
		&- \left\| g_{k,\alpha_k} - \grad f(p_{k})  \right\| - \left\| (\B_{k,\alpha_k} -\hess f(p_k))[v_{k} ]\right\|  -  \sigma_{k+1} \|v_{k} \|^2.
	\end{align*}
        Rearranging, using  \eqref{eq:Ass3-lc} with $v=v_{k}$ and \eqref{eq:grahess.alg} with $\alpha=\alpha_k$, we get
	\begin{align*}
		\|  {\cal G}(p_k,v_{k})  \|
		 &\leq    \frac{L' + 2(\theta + \sigma_{k+1} )}{2} \| v_{k} \|^2 + \frac{ \kappa_{g} }{2^{\alpha_k - 1} } \| v_{k-1} \|^2 + \left\| \B_{k,\alpha_k} -\hess f(p_k) \right\|_{\op} \|v_{k}\|  \\
		&\leq   \frac{L' + 2(\theta+\sigma_{k+1})}{2} \| v_{k} \|^2 + \frac{ \kappa_{g} }{2^{\alpha_k - 1}} \| v_{k-1} \|^2 +   \frac{ \kappa_{\B} }{ 2^{\alpha_k - 1} } \| v_{k-1} \| \|v_{k}\|  \\	
		&\leq   \frac{L' + 2(\theta+\sigma_{k+1})}{2} \| v_{k} \|^2 + 2 \kappa_{g}  \| v_{k-1} \|^2 +   2 \kappa_{\B}  \| v_{k-1} \| \|v_{k}\|  \\
		&\leq
		\frac{L' + 2(\theta+\sigma_{k+1}) +  4( \kappa_g + \kappa_{\B}) }{2}
		\max\left\{  \| v_{k-1} \| , \|v_{k}\|   \right\}^2.
	\end{align*}
	Therefore, the proof conclusion follows from the previous inequality together with Corollary \ref{cor:gooddef}.
	\end{proof}

Now we can establish our first complexity result for Algorithm 1. Here, we are concerned only with the particular case where the retraction chosen is the exponential application, i.e., $\R = \exp$, and Assumption~\ref{Ass3} for ${\cal G} \colon \T{\cal M} \to \T{\cal M}$ defined by
	\begin{equation}\label{eq:cG-exp}
	{\cal G}(p,v) = \p_{v}^{-1} \grad f(\exp_{p}v) \,\, \mbox{ for all } \, (p,v) \in \T {\cal M}.
	\end{equation}

\begin{theorem}
	\label{thm:main.first-order}
	Let $\R= \exp$. Under Assumption~\ref{Ass1}, \ref{Ass2}, and \ref{Ass3} with ${\cal G}$ given in \eqref{eq:cG-exp}, let $p_0, p_1,p_2,\ldots$ be the iterates produced by Algorithm 1. For arbitrary $\epsilon > 0$, if  $\|\grad f(p_{k})\|>\epsilon$ for all  $ k \in \{1,\ldots,N\}$ then
	\begin{equation}\label{eq:Nthemaexp.fo}
		 N \leq 2 +  \left[\frac{48(f(p_{1})-f_{low})}{\sigma_{1}}+2 \|v_0\|^3 \right] \left( \frac{\epsilon}{\tau} \right)^{-\frac{3}{2}}
	\end{equation}
	where $\tau$ is defined in \eqref{eq:tegdgbgt}.
\end{theorem}

\begin{proof}
Inequality \eqref{eq:Nthemaexp.fo} is trivially satisfied for $N=1$ and $N=2$. Then assume $N\geq 3$. Taking into account the hypothesis of this theorem and the fact that
$\bar{k}$ defined in \eqref{eq:defkbar} belongs to $\{1,\ldots,N-2\}$, we conclude that $\|\grad f(p_{\bar{k}+2})\|>\epsilon$. Hence, since $\p_{v_{ \bar{k} + 1 }}^{-1}$ is an isometry,  Lemma~\ref{lem:bound.calG} with
 $k=\bar{k} + 1$ and ${\cal G}$ as in \eqref{eq:cG-exp} yields
	$$
	\epsilon  <  \|\grad f(p_{\bar{k}+2} )\|
	              = \| \p_{v_{\bar{k}+1}}^{-1} \grad f(\exp_{p_{\bar{k}+1}}v_{\bar{k}+1}) \| = \| {\cal G}(p_{\bar{k}+1} , v_{\bar{k}+1 }) \| \leq \tau \max\left\{ \| v_{ \bar{k} } \|  , \| v_{ \bar{k} +1 } \| \right\}^{2},
	$$
  with $\tau$ given in \eqref{eq:tegdgbgt}. By using this inequality together with Lemma~\ref{lem:aux.conv.expc} we find
  $$
 \epsilon <  \tau \left[\frac{48(f(p_{1})-f_{low})}{\sigma_{1}}+2 \|v_0\|^3 \right]^{\frac{2}{3}}\dfrac{1}{(N-2)^{\frac{2}{3}}}.
 $$
	After rearranging the terms of this inequality, the proof will be complete.
\end{proof}

The proof of the next result follows directly from the previous theorem and  the fact that Assumption~\ref{Ass2} and \ref{Ass3} are satisfied for $\R=\exp$, $L' = L$ and ${\cal G}$ given in \eqref{eq:cG-exp} for $L$-Lipschitz $\hess f$.

\begin{corollary}\label{cor:poexp}
Under Assumption~\ref{Ass1}, let $p_0, p_1,p_2,\ldots$ be the iterates produced by Algorithm 1 with $\R= \exp$. For arbitrary $\epsilon > 0$, with  L-Lipschitz $\hess f$ $\|\grad f(p_{k})\| \leq \epsilon$ for all
	$$
	 k > 2 +  \left[\frac{48(f(p_{1})-f_{low})}{\sigma_{1}}+2 \|v_0\|^3 \right] \left( \frac{\epsilon}{\tau} \right)^{-\frac{3}{2}},
	$$
	where $\tau$ is considered with $L'=L$ in \eqref{eq:tegdgbgt}. In particular, $\lim_{k \to \infty } \| \grad f (p_k) \| =0 $.
\end{corollary}


Our second complexity result will be proposed for a general retraction $\R$ whose proof requires the use of an additional assumption that relates $\nabla \hat{f}_k$ and $\grad f$ for each iteration $k$. This assumption is detailed below.

\begin{assumption}\label{Ass4}
	 There exist constants $a\in(0,+\infty]$ and $b\in(0,1]$ such that	
            $\|v\|\leq a $,  $v\in \T_{p_k} {\cal M}$, implies $ \|   \nabla \hat{f}_k ( v ) \| \geq b \| \grad f(\R_{p_k}(v)) \|$.
	\end{assumption}
	
	\begin{remark}
	For the first-order analysis of a general retraction  the following assumption was used in \cite{agarwal2021adaptive}:
	\begin{itemize}
	\item[ \textbf{(A)} ]
	There exist constants $a\in(0,+\infty]$ and $b\in(0,1]$ such that	
            $\|v\|\leq a $,  $v\in \T_{p_k} {\cal M}$, implies  $ \varsigma_{\min}( \DR_{p_k}(v) ) \geq b$, where $\varsigma_{\min}$ extracts the smallest singular value of an operator.
	\end{itemize}
	Since the calculations in \cite[Sect. 4]{agarwal2021adaptive}  guarantee the inequality
	 \begin{equation*}\label{eq:relationReE}
	\|   \nabla \hat{f}_k ( v ) \| \geq \varsigma_{\min} (\DR_{p_k}(v)) \| \grad f(\R_{p_k}(v)) \|
	\end{equation*}
        for all $v\in \T_{p_k} {\cal M}$, it is easy to show that if \textbf{(A)} holds then Assumption~\ref{Ass4} also holds. In view of this, we can state that
       \cite[Sect. 7]{agarwal2021adaptive} secures Assumption~\ref{Ass4} for a large family of manifolds and retractions.
	\end{remark}

We now prepare to present our first-order complexity result for a general retraction $\R$. Here, we consider the function ${\cal G} \colon \T{\cal M} \to \T{\cal M}$ defined by
       \begin{equation}\label{eq:cG-r}
       {\cal G} (p,v) = \nabla  (f\circ\R_p)(v) \,\, \mbox{ for all } \, (p,v) \in \T {\cal M}.
       \end{equation}

\begin{theorem}\label{thm:main.first-order-r}
Under Assumptions~\ref{Ass1}, \ref{Ass2}, \ref{Ass3} with ${\cal G}$ given in \eqref{eq:cG-r}, and \ref{Ass4}, let $p_0, p_1,p_2,\ldots$ be the iterates produced by Algorithm 1. Choose $a\in (0,+\infty]$ and $b\in (0,1]$  satisfying Assumption~\ref{Ass4}. For every $\epsilon > 0$, if  $\|\grad f(p_{k})\|>\epsilon$ for all  $ k \in \{1,\ldots,N\}$ then
\begin{equation}\label{eq:main.first-order-r}
		N \leq 2 +  \left[\frac{48(f(p_{1})-f_{low})}{\sigma_{1}}+2 \|v_0\|^3 \right]   \max \left\{ a^{-3} , \left( \frac{b\epsilon}{ \tau } \right)^{-\frac{3}{2}} \right\},
	\end{equation}
	where $\tau$ is defined in \eqref{eq:tegdgbgt}.
\end{theorem}
\begin{proof}
Inequality \eqref{eq:main.first-order-r} is trivially satisfied when $N=1$ and $N=2$. Assume $N\geq 3$. Throughout the proof consider the definition of $\bar{k}$ given in \eqref{eq:defkbar}. Our analysis will be divided into the following two cases:
\begin{itemize}
 \item[(i)] $\max \{ \| v_{ \bar{k}  } \| , \| v_{\bar{k} + 1}\| \} \in [0,a)$;
 \item[(ii)] $\max\{ \| v_{ \bar{k}  } \| , \| v_{\bar{k} + 1 }\| \} \geq a$.
 \end{itemize}
 If case (i) holds then $\| v_{\bar{k}+1} \|\leq a$ and, by Assumption~\ref{Ass4} with $v=v_{\bar{k}+1}$ and $k=\bar{k}+1$, we have
$$
\|   \nabla \hat{f}_{\bar{k}+1} ( v_{\bar{k}+1} ) \| \geq b \| \grad f(\R_{p_{\bar{k}+1}}(v_{\bar{k}+1})) \| = b\|\grad f(p_{\bar{k}+2})\|.
$$
By hypothesis and the fact that $\bar{k} \in \{1,\ldots,N-2\}$ we conclude that
 $ \|\grad f(p_{\bar{k}+2})\|>\epsilon$ and, therefore, the above inequality leads to $\|   \nabla \hat{f}_{\bar{k}+1} ( v_{\bar{k}+1} ) \| > b\epsilon $.
 Using this, \eqref{eq:cG-r} with $(p,v)=(p_{\bar{k}+1} , v_{\bar{k}+1} )$ and
 Lemma~\ref{lem:bound.calG} with $k=\bar{k}+1$ and ${\cal G}$ as in \eqref{eq:cG-r}, one can conclude that
$$
\left( \frac{b\epsilon}{\tau}  \right)^{ \frac{1}{2} } <  \left( \frac{ \|   \nabla \hat{f}_{\bar{k}+1} ( v_{\bar{k}+1 } ) \| }{\tau} \right)^{ \frac{1}{2} } = \left( \frac{ \| {\cal G}(p_{\bar{k}+1}, v_{ \bar{k}+1 } ) \| }{\tau} \right)^{ \frac{1}{2} }
                 \leq  \max\left\{ \| v_{ \bar{k} } \|  , \| v_{ \bar{k} +1 } \| \right\}.
$$
On the other hand, if case (ii) holds then
         $
	 a \leq \max\left\{ \| v_{\bar{k}-1} \| , \| v_{\bar{k}} \|\right\}.
	 $
	   Thus, in all cases, one can conclude that
   $$
   \min\left\{ a , \left( \frac{b\epsilon}{ \tau } \right)^{\frac{1}{2}} \right\}
   \leq \max\left\{ \| v_{\bar{k}-1} \| , \| v_{\bar{k}} \|\right\}
   \leq \left[\frac{48(f(p_{1})-f_{low})}{\sigma_{1}}+2 \|v_0\|^3 \right]^{\frac{1}{3}}\dfrac{1}{(N-2)^{\frac{1}{3}}},
   $$
where the second inequality follows from Lemma \ref{lem:aux.conv.expc}. By rearranging the terms in a convenient way we can get \eqref{eq:main.first-order-r}.
\end{proof}

As previously discussed in Remarks~\ref{rem:Ass2} and \ref{rem:Ass3}, if $\nabla^2 \hat{f}_k$ is $L$-Lipschitz for each iteration $k$, and $\R$ is a second-order retraction, then Assumptions~\ref{Ass2} and \ref{Ass3} are satisfied with $L' = L$ and ${\cal G}$
 given in \eqref{eq:cG-r}. Considering this and the previous theorem, we can derive the following corollary.

\begin{corollary}\label{cor:poR}
Under Assumptions~\ref{Ass1} and ~\ref{Ass4}, let $p_0, p_1,p_2,\ldots$ be the iterates produced by Algorithm 1 with a second-order retraction chosen. In each iteration $k$ assume that $\nabla^2 \hat{f}_k$ is L-Lipschitz. Choose $a\in (0,+\infty]$ and $b\in (0,1]$ satisfying Assumption~\ref{Ass4}. Then, for every $\epsilon > 0$, one has $\|\grad f(p_{k})\| \leq \epsilon$ for all
 \begin{equation*}\label{eq:main.first-order-r-co}
		k >2 +  \left[\frac{48(f(p_{1})-f_{low})}{\sigma_{1}}+2 \|v_0\|^3 \right]   \max \left\{ a^{-3} , \left( \frac{b\epsilon}{ \tau } \right)^{-\frac{3}{2}} \right\},
	\end{equation*}
	where $\tau$ is considered with $L'=L$ in \eqref{eq:tegdgbgt}. In particular, $\lim_{k \to \infty } \| \grad f (p_k) \| =0 $.
\end{corollary}


         \subsection{Second-order analysis}

 In this subsection we shall give a second-order complexity result for Algorithm 1 with second-order progress condition \eqref{eq:icrehe3} enforced. This condition is similar to one given in \cite{grapiglia2022cubic} for the same purpose in the Euclidean case. Unlike the first-order analysis in the previous subsection,  here we give second-order analysis only for general retraction $\R$, but not for the particular $\R = \exp$.

\begin{theorem}\label{thm:main.hess}
	Under Assumptions~\ref{Ass1} and \ref{Ass2}, let $p_{0}, p_1 \ldots $ be the iterates produced by Algorithm 1 with second-order progress \eqref{eq:icrehe3} enforced. For every  $\epsilon > 0$, if
        $\lambda_{\min} \left( \hess f(p_{k}) \right)< - \epsilon $ for all  $k \in \{1,\ldots,N\}$ then
        \begin{equation}\label{tep:seceppghn.}
           N     \leq   2 +
    \left[\frac{48(f(p_{1})-f_{low})}{\sigma_{1}}+2 \|v_0\|^3 \right]  \left[ \frac{ \epsilon }{ \sigma_{ \max }  + \theta + \kappa_{\B} } \right]^{-3},
 	\end{equation}
	where $\sigma_{\max}$ is defined in \eqref{eq:bound.sigmak}.
\end{theorem}

\begin{proof}
 It is clear that \eqref{tep:seceppghn.} holds for $N=1$ and $N=2$. Assume $N\geq 3$.  Let $\bar{k}$ be the constant defined in \eqref{eq:defkbar}. By using the second inequality of \eqref{eq:grahess.alg} with $k=\bar{k}+1$ and $\alpha=\alpha_{\bar{k}+1}$, we obtain
         \begin{align*}
         \left\langle  \B_{ \bar{k}+1, \alpha_{ \bar{k} +1 } } \left[ \frac{v}{\|v\|} \right], \frac{v}{\|v\|} \right\rangle
         &= \left\langle  \hess f(p_{\bar{k} +1 }) \left[ \frac{v}{\|v\|} \right], \frac{v}{\|v\|} \right\rangle + \left\langle \left( \B_{ \bar{k} + 1, \alpha_{ \bar{k} +1 } }  -  \hess f(p_{\bar{k} + 1 })  \right)\left[ \frac{v}{\|v\|} \right], \frac{v}{\|v\|} \right\rangle  \\
		&\leq \left\langle  \hess f(p_{\bar{k} +1 }) \left[ \frac{v}{\|v\|} \right], \frac{v}{\|v\|} \right\rangle + \left\| \B_{ \bar{k} +1, \alpha_{ \bar{k} +1 } } -  \hess f(p_{\bar{k} +1 })  \right\|_{\op}\\
		&\leq \left\langle  \hess f(p_{\bar{k} +1  }) \left[ \frac{v}{\|v\|} \right], \frac{v}{\|v\|} \right\rangle + \frac{\kappa_{\B}}{2^{\alpha-1}} \| v_{ \bar{k}  } \| \\
		&\leq \left\langle  \hess f(p_{\bar{k} +1  }) \left[ \frac{v}{\|v\|} \right], \frac{v}{\|v\|} \right\rangle + 2\kappa_{\B} \| v_{ \bar{k}  } \|,
	\end{align*}
	for all $  v\in \T_{p_{ \bar{k} +1 }} {\mathcal{M}} \backslash \{0\}$.  After calculating the infimum with respect to $  v\in \T_{p_{ \bar{k} +1  }} {\mathcal{M}}$ on both sides and  use the hypothesis of this theorem, we arrive at
	$$
		\lambda_{\min}\left( \B_{ \bar{k} + 1 , \alpha_{ \bar{k} + 1 }  } \right) \leq \lambda_{\min} \left( \hess f(p_{\bar{k} +1  }) \right)
 + 2\kappa_{\B} \| v_{ \bar{k}  } \|
 < - \epsilon
 + 2 \kappa_{\B} \| v_{ \bar{k}  } \| \leq - \epsilon
 + 2 \kappa_{\B} \max\{  \| v_{ \bar{k}  } \|  , \| v_{ \bar{k} + 1 } \|   \}.
	$$
	It follows from $ \sigma_{ \bar{k} + 2 } = 2^{\alpha_{ \bar{k} +1 } - 1} \sigma_{ \bar{k}  + 1}  $, $v_{\bar{k}+1}=v_{\bar{k} +1, \alpha_{\bar{k}+1}}$, Corollary~\ref{cor:gooddef} with the previous inequalities, \eqref{eq:icrehe3} with $k=\bar{k}+1$ and $\alpha=\alpha_{ \bar{k} +1 }$ that
\begin{align*}
         - (\sigma_{ \max }  + \theta ) \max\{\| v_{ \bar{k} } \|, \| v_{ \bar{k} + 1} \| \}
         & \leq  - (   \sigma_{ \bar{k} + 2}  + \theta ) \max\{\| v_{ \bar{k} } \|, \| v_{ \bar{k} +1 } \| \}  \\
         &= - (   2^{\alpha_{ \bar{k} +1 }-1}\sigma_{\bar{k} +1 }  + \theta ) \max\{\| v_{\bar{k}} \|, \| v_{\bar{k} +1,\alpha_{\bar{k} + 1}} \| \}   \\
          &  \leq -   2^{\alpha_{ \bar{k} +1 }-1}\sigma_{\bar{k} +1 }  \| v_{\bar{k} +1,\alpha_{\bar{k} + 1}} \|  -  \theta \| v_{\bar{k}} \| \\
        & \leq  \lambda_{\min}\left( \B_{ \bar{k}+1, \alpha_{ \bar{k}+1 } } \right)
                 <     - \epsilon
 + 2 \kappa_{\B} \max\{  \| v_{ \bar{k}  } \|  , \| v_{ \bar{k} +1 } \|   \}.
	\end{align*}
       Hence, it follows that
           $$
             \frac{ \epsilon }{ \sigma_{ \max }  + \theta + 2\kappa_{\B} }    <
  \max\{  \| v_{ \bar{k}  } \|  , \| v_{ \bar{k} +1  } \|   \}  \leq \left[\frac{48(f(p_{1})-f_{low})}{\sigma_{1}}+2 \|v_0\|^3 \right]^{\frac{1}{3}} \frac{1}{(N-2)^{\frac{1}{3}}},
       $$
	where the second inequality comes from Lemma~\ref{lem:aux.conv.expc}. Starting from this inequality, we arrive at \eqref{tep:seceppghn.} by using simple algebraic manipulations. Therefore, the proof is complete.
\end{proof}

The next result is an immediate consequence of the previous theorem. Looking at this result and at Corollary~\ref{cor:poexp} or \ref{cor:poR}, depending on the Assumptions required, it is possible to find a number of iterations $N$ such that $p_N$ is an approximate second-order critical point, that is $p_N$ satisfies $\|\grad f(p_N) \| \leq \epsilon_g$ and $\lambda_{\min} \left( \hess f(p_{N}) \right) \geq - \epsilon_{H} $, for $\epsilon_{g}>0$ and $\epsilon_{H}>0$ chosen arbitrarily.

\begin{corollary}\label{cor:main.hess.cor}
	Under Assumptions~\ref{Ass1} and \ref{Ass2}, let $p_{0}, p_1 \ldots $ be the iterates produced by Algorithm 1 with second-order progress \eqref{eq:icrehe3} enforced. Then, for every  $\epsilon > 0$, one has
        $\lambda_{\min} \left( \hess f(p_{k}) \right) \geq - \epsilon $ for all
        \begin{equation*}
           k     >    2 +
    \left[\frac{48(f(p_{1})-f_{low})}{\sigma_{1}}+2 \|v_0\|^3 \right]  \left[ \frac{ \epsilon }{ \sigma_{ \max }  + \theta + \kappa_{\B} } \right]^{-3},
 	\end{equation*}
	where $\sigma_{\max}$ is defined in \eqref{eq:bound.sigmak}. In particular,
	$\liminf_{ k \to \infty } \lambda_{\min} \left( \hess f(p_{k}) \right)  \geq 0$.
\end{corollary}


\section{Finite-Difference Approximations for Gradients and Hessians}\label{sec:FDAGH}

 In this section we borrow finite-differences to generate
 approximations $g_{k,\alpha}$ and $\B_{k,\alpha}$ satisfying \eqref{eq:grahess.alg}. These approximations are Riemannian extensions of those already known in the Euclidean case, see \cite{cartis2012oracle,nocedal1999numerical} for instance.

 For results established in this section, we assume that, for each iteration
  $k$ and constant $\alpha \geq 0$, $\mathfrak{B}^k \coloneqq \{ e^k_1, \ldots, e^k_n \}$ forms an orthonormal basis for $\T_{p_k} {\cal M}$ and  $h_{k,\alpha}$ is a real number defined by
\begin{equation}\label{eq:defhka}
h_{k,\alpha} \coloneqq \frac{\| v_{k-1} \|}{2^{\alpha - 1}\sigma_{k}}.
\end{equation}




\subsection{Approximation for Gradients}

The following result provides an approximation
 $g_{k,\alpha}$ for $\grad f(p_k)$ that depends only on evaluations of the objective function $f$ and satisfies the first inequality in \eqref{eq:grahess.alg} for every iteration $k$ and constant $\alpha \geq 0$.   The analysis presented here is developed for a general retraction $\R$. Recall that
$\hat{f}_k$ refers to the notation introduced in \eqref{notation}.

\begin{proposition}\label{prop:g-grad-apro}
      For every iteration $k$ and constant $\alpha \geq 0$, let $g_{k,\alpha} \in   \T_{p_k} {\cal M}$ be defined by
	\begin{equation}\label{eq:algAdehgr}
		g_{k,\alpha} = \sum_{i=1}^{n} \frac{ \hat{f}_k( h_{k,\alpha} e_{i}^{k} )  -  \hat{f}_k( - h_{k,\alpha} e_{i}^{k} )  }{ 2 h_{k,\alpha} } e_{i}^{k}.
	\end{equation}
Under Assumption~\ref{Ass2}, the first inequality of \eqref{eq:grahess.alg} is always satisfied with $\kappa_{g} = L\sqrt{n}/(3 \sigma^2_1)$.
\end{proposition}
\begin{proof}
By leveraging the orthonormality of the basis $\mathfrak{B}^k = \{ e^k_1,\ldots, e^k_n \}$, we can assert that
\begin{align}\label{eq:pepagrad}
\left\| g_{k,\alpha} - \grad f(p_k) \right\|
&\leq \sqrt{n} \max_{i=1,\ldots,n} | \langle g_{k,\alpha} - \grad f(p_k) , e_{i}^{k} \rangle | \nonumber \\
&= \frac{ \sqrt{n} }{ 2 h_{k,\alpha} }\max_{i=1,\ldots,n} | 2 h_{k,\alpha} \langle g_{k,\alpha} - \grad f(p_k) , e_{i}^{k} \rangle |.
\end{align}
Use \eqref{eq:algAdehgr} to rewrite the term within the modulus that appears in
  \eqref{eq:pepagrad} as follows:
\begin{align*}
 \MoveEqLeft
 2 h_{k,\alpha} \left\langle g_{k,\alpha} - \grad f(p_k) , e_{i}^{k} \right\rangle  \\
 & =  \hat{f}_k( h_{k,\alpha} e_{i}^{k} ) - f(p_k) - \langle\grad f(p_k), h_{k,\alpha} e_{i}^{k} \rangle - \dfrac{1}{2}\langle\hess f(p_k)[ h_{k,\alpha} e_{i}^{k} ], h_{k,\alpha} e_{i}^{k} \rangle  \\
& \quad - \left( \hat{f}_k( - h_{k,\alpha} e_{i}^{k} ) - f(p_k) - \langle\grad f(p_k), (-h_{k,\alpha} e_{i}^{k} ) \rangle - \dfrac{1}{2}\langle\hess f(p_k)[ - h_{k,\alpha} e_{i}^{k} ], (- h_{k,\alpha} e_{i}^{k} ) \rangle \right).
 \end{align*}
Applying the triangle inequality and utilizing Assumption~\ref{Ass2} with $v =h_{k,\alpha} e_{i}^{k}$ and $v=-h_{k,\alpha} e_{i}^{k}$, we obtain
\begin{align*}
 \MoveEqLeft
 \left| 2 h_{k,\alpha} \left\langle g_{k,\alpha} - \grad f(p_k) , e_{i}^{k} \right\rangle  \right| \\
& \leq \left| \hat{f}_k( h_{k,\alpha} e_{i}^{k} ) - f(p_k) - \langle\grad f(p_k), h_{k,\alpha} e_{i}^{k} \rangle - \dfrac{1}{2}\langle\hess f(p_k)[ h_{k,\alpha} e_{i}^{k} ], h_{k,\alpha} e_{i}^{k} \rangle \right| \\
& \quad + \left| \hat{f}_k( - h_{k,\alpha} e_{i}^{k} ) - f(p_k) - \langle\grad f(p_k), (-h_{k,\alpha} e_{i}^{k} ) \rangle - \dfrac{1}{2}\langle\hess f(p_k)[ - h_{k,\alpha} e_{i}^{k} ], (- h_{k,\alpha} e_{i}^{k} ) \rangle \right| \\
& \leq \dfrac{L}{6}\| h_{k,\alpha} e_{i}^{k} \|^{3} + \dfrac{L}{6}\| - h_{k,\alpha} e_{i}^{k} \|^{3} =  \dfrac{L}{3} (h_{k,\alpha})^{3}.
\end{align*}
From this inequality, along with \eqref{eq:pepagrad}, Corollary~\ref{cor:gooddef}, and \eqref{eq:defhka}, it follows that
$$
\| g_{k,\alpha} - \grad f(p_k) \|
\leq \frac{ L \sqrt{n} }{ 6  } ( h_{k,\alpha} )^2
=     \frac{ L \sqrt{n} }{ 6  (2^{\alpha -1} \sigma_k^2)}  \frac{\| v_{k-1} \|^2}{ 2^{\alpha -1} }
\leq \frac{ L \sqrt{n} }{ 3 \sigma_1^2}  \frac{\| v_{k-1} \|^2}{ 2^{\alpha -1} } .
$$
This completes the proof.
\end{proof}


\subsection{Approximations for Hessians}

Here, we present some approximations $\B_{k,\alpha}$ satisfying the second inequality in \eqref{eq:grahess.alg}. The following result provides a $\B_{k,\alpha}$ inspired by the finite difference approximation of the Hessian proposed in \cite{absil2008optimization}. This approximation relies on the evaluation of $\grad f$ and the choice of a mapping ${\cal G} \colon \T{\cal M} \to \T{\cal M}$ associated with Assumption \ref{Ass3}. In Remark 4, two natural options for choosing ${\cal G}$ are proposed, and their convenience of use depends on the retraction chosen in Algorithm 1.

\begin{proposition} \label{prop:B-Hess}
	Let ${\cal G} \colon \T{\cal M} \to \T{\cal M}$ be a mapping such that the image of ${\cal G}(p, \cdot ) \colon \T_p{\cal M} \to \T{\cal M}$ belongs to $\T_p{\cal M}$ for all $p\in {\cal M}$.
Suppose that Assumption~\ref{Ass3} is satisfied. For each iteration $k$ and constant $\alpha \geq 0$, if $\B_{k,\alpha} \colon  \T_{p_k} {\cal M} \to  \T_{p_k} {\cal M}$ is defined as
	\begin{equation}\label{eq:algBsaac}
		\B_{k,\alpha} = \frac{  \A_{k,\alpha}  +  \A_{k,\alpha}^* }{2},
	\end{equation}
	where $\A_{k,\alpha} \colon  \T_{p_k} {\cal M} \to  \T_{p_k} {\cal M}$ is the operator that, for each $i \in \{ 1, \ldots , n \}$, assumes
	\begin{equation}\label{eq:algAdeh}
		\A_{k,\alpha}[e^k_i]= \frac{ {\cal G}(p_{k} , h_{k,\alpha} e^{k}_{i})-\grad f(p_k)}{h_{k,\alpha}},
		\end{equation}
         then the second inequality in \eqref{eq:grahess.alg} is always satisfied with
          $\kappa_{\B} = \sqrt{n} L' / (2 \sigma_{1} )  $.
\end{proposition}
\begin{proof}
 It follows from the definition of $\| \cdot \|_{\op}$ along with the orthonormality of the basis $\mathfrak{B}^k$, self-adjoint $\hess f(p_k)$ and $\B_{k,\alpha}$  and \eqref{eq:algBsaac} that
     \begin{align*}
	      \MoveEqLeft
		\left\| \B_{k,\alpha}-\hess f(p_k)  \right\|_{\op}  \\
		& =  \sup\{  \| \left( \B_{k,\alpha}-\hess f(p_k) \right) [ v ] \| \colon  v\in \T_{p_k}{\cal M}, \, \|v\|=1  \}\\
		& =  \sup\left\{  \left\| \sum_{i=1}^{n} \left\langle \left( \B_{k,\alpha}-\hess f(p_k) \right)  [v], e_{i}^{k}  \right\rangle e_{i}^{k}  \right\| \colon  v\in \T_{p_k}{\cal M}, \, \|v\|=1  \right\}\\
		& =  \sup\left\{  \left( \sum_{i=1}^{n} \left\langle \left( \B_{k,\alpha}-\hess f(p_k) \right)  [v], e_{i}^{k}  \right\rangle^2 \right)^{\frac{1}{2}}  \colon  v\in \T_{p_k}{\cal M}, \, \|v\|=1  \right\}\\
		& =  \sup\left\{  \left( \sum_{i=1}^{n} \left\langle \left( \B_{k,\alpha}-\hess f(p_k) \right)  [e_{i}^{k}] , v \right\rangle^2 \right)^{\frac{1}{2}}  \colon  v\in \T_{p_k}{\cal M}, \, \|v\|=1  \right\}\\
		& \leq  \sup\left\{  \left( \sum_{i=1}^{n} \left\| \left( \B_{k,\alpha}-\hess f(p_k) \right)  [e_{i}^{k}] \right\|^2 \| v \|^2 \right)^{\frac{1}{2}}  \colon  v\in \T_{p_k}{\cal M}, \, \|v\|=1  \right\}\\
		& \leq  \sqrt{n} \max_{i=1,\ldots,n}  \left\| \left( \B_{k,\alpha}-\hess f(p_k) \right)  [e_{i}^{k}] \right\|.
	\end{align*}
	Furthermore, from \eqref{eq:algBsaac}, it comes that
		\begin{align*}
		\left\| \left( \B_{k,\alpha} - \hess f(p_k) \right)[e_{i}^{k}] \right\| & = \frac{1}{2}
		\left\| \left( \A_{k,\alpha} - \hess f(p_k) \right) [e_{i}^{k}]   +  \left( \A_{k,\alpha} -\hess f(p_k) \right)^*[e_{i}^{k}]   \right\| \\
                & = \frac{1}{2}
		\left(  \left\| \left( \A_{k,\alpha} - \hess f(p_k) \right) [e_{i}^{k}]  \right\| + \left\|  \left( \A_{k,\alpha} -\hess f(p_k) \right)^*[e_{i}^{k}]   \right\| \right) \\
	       & =\left\| \left( \A_{k,\alpha}-\hess f(p_k) \right) [e_{i}^{k}] \right\|,
	\end{align*}
	for all $i \in \{ 1, \ldots , n \}$.
 Therefore, considering Assumption~\ref{Ass3} with $v=h_{k,\alpha}e^k_i$, \eqref{eq:defhka}, \eqref{eq:algAdeh} and Corollary~\ref{cor:gooddef}, we obtain
        \begin{align*}
		\left\| \B_{k,\alpha}-\hess f(p_k)\right\|_{\op}
		& \leq  \sqrt{n} \max_{i=1,\ldots,n}  \left\|( \A_{k,\alpha}-\hess f(p_k))[e^k_{i}] \right\| \\
		& = \frac{ \sqrt{n} }{ h_{k,\alpha} } \max_{i=1,\ldots,n}  \left\| {\cal G}(p_k,h_{k,\alpha} e^k_i) - \grad f(p_k)         - \hess f(p_k)[h_{k,\alpha}e^k_i] \right\| \\
		&\leq \frac{ \sqrt{n} L'}{ 2h_{k,\alpha} } \max_{i=1,\ldots,n}  \left\| h_{k,\alpha}e^k_i \right\|^2  = \frac{ \sqrt{n} L' }{ 2 \sigma_{k} } \frac{\| v_{k-1} \|}{2^{\alpha-1}}  \leq  \frac{ \kappa_{\B} }{2^{\alpha-1}} \| v_{k-1} \|.
	\end{align*}
	This completes the proof.
\end{proof}

Unlike the previous result, the approximation $\B_{k,\alpha}$ provided in the following result relies solely on evaluations of the objective function
$f$ but not the evaluation of the gradient of $f$. 

\begin{proposition} \label{lem:B-Hess.sg}
Consider Algorithm 1 with a general second-order retraction (case 1) and with $\R=\exp$ (case 2). For each iteration $k$ and constant $\alpha \geq 0$, let $\B_{k,\alpha} \colon  \T_{p_k} {\cal M} \to  \T_{p_k} {\cal M}$ be the operator defined in \eqref{eq:algBsaac} with $\A_{k,\alpha} \colon  \T_{p_k} {\cal M} \to  \T_{p_k} {\cal M}$ characterized by
\begin{equation}\label{eq:mappauxGc1c2}
	\langle \A_{k,\alpha}[e^k_i] , e^k_j \rangle =
	\begin{cases}
	\frac{  \hat{f}_k( h_{k,\alpha} e^{k}_{i} + h_{k,\alpha} e^{k}_{j} )
		- \hat{f}_k( h_{k,\alpha} e^{k}_{i} )
		- \hat{f}_k( h_{k,\alpha} e^{k}_{j} )
		+\hat{f}_k(0) }{(h_{k,\alpha})^2},  \quad \mbox{for case 1},
		\\
		\\
		\frac{  f \left( \exp_{q_{k,\alpha}^{i} } \left(  \p_{ v_{k,\alpha}^{i} } v_{k,\alpha}^{j}  \right) \right)
		- f( q_{k,\alpha}^{i} )
		- f( q_{k,\alpha}^{j} )
		+f(p_k) }{(h_{k,\alpha})^2}, \quad \mbox{for case 2},
			\end{cases}
\end{equation}
         for every $i,j \in \{ 1, \ldots , n \}$, where the notations
          $q_{k,\alpha}^{i}   \coloneqq \exp_{p_k} v_{k,\alpha}^{i} $ and
		$v_{k,\alpha}^{i} \coloneqq h_{k,\alpha} e^{k}_{i}$ are assumed for every $i \in \{ 1, \ldots , n \}$. Additionally, assume that $\nabla^2 \hat{f}_k$ is $L$-Lipschitz for case 1, and assume that $\hess f$ is $L$-Lipschitz for case 2. Then the second inequality of \eqref{eq:grahess.alg} is always satisfied for $ \kappa_{\B} = L( 5n + 3\sqrt{n} )/6 \sigma_1 $ in both cases.
\end{proposition}
\begin{proof}
Take an iteration $k$ and a constant $\alpha\geq 0$. By following the idea of the proof of Proposition \ref{prop:B-Hess}, we can conclude that
        \begin{align}
		\left\| \B_{k,\alpha}-\hess f(p_k)\right\|_{\op}
		& \leq  \sqrt{n} \max_{i=1,\ldots,n}  \left\|( \A_{k,\alpha}-\hess f(p_k))[e^k_{i}] \right\| \nonumber \\
		& = \frac{ \sqrt{n} }{ h_{k,\alpha} } \max_{i=1,\ldots,n}  \left\| h_{k,\alpha} \A_{k,\alpha}[e^k_{i}] -\hess f(p_k)[v_{k,\alpha}^{i}] \right\| \label{eq:proisag}
	\end{align}
	for both cases. Given that $\nabla^2 \hat{f}_k$ is $L$-Lipschitz for case 1 and $\hess f$ is $L$-Lipschitz for case 2, it follows from Remark~\ref{rem:Ass3} with $v = v_{k,\alpha}^{i} $ that
	\begin{equation}\label{eq:Ass3-lc.prop3c}
\left\| {\cal G}( p_k ,v_{k,\alpha}^{i}  )  - \grad f(p_k) - \hess f(p_k)[ v_{k,\alpha}^{i} ]  \right\| \leq\dfrac{L}{2}\| v_{k,\alpha}^{i} \|^{2},
        \end{equation}
        where
        \begin{equation}\label{eq:cG-exp.c1ec2}
	{\cal G}( p_k ,v_{k,\alpha}^{i}  ) = 	
	\begin{cases}
	\nabla \hat{f}_k( v_{k,\alpha}^{i}  ),  \quad \mbox{for case 1},
		\\
		\\
	\p_{ v_{k,\alpha}^{i}  }^{-1} \grad f(\exp_{p_k} v_{k,\alpha}^{i}  ), \quad \mbox{for case 2}.
	\end{cases}
	\end{equation}
	Then, by applying \eqref{eq:Ass3-lc.prop3c}, along with the definition of $v_{k,\alpha}^{i}$ and the orthonormality of $\mathfrak{B}^k$, we find
	\begin{align}\label{eq:pro.reto(i)}
	\MoveEqLeft
	\left\| h_{k,\alpha} \A_{k,\alpha}[e^k_{i}] - \hess f(p_k)[ v_{k,\alpha}^{i} ]  \right\| \nonumber\\
	& = \left\| h_{k,\alpha} \A_{k,\alpha}[e^k_{i}] - {\cal G}( p_k ,v_{k,\alpha}^{i} )
	+ \grad f(p_k)
	+ {\cal G}( p_k ,v_{k,\alpha}^{i} ) - \grad f(p_k) -\hess f(p_k)[ v_{k,\alpha}^{i} ]\right\| \nonumber  \\
	& \leq  \left\| h_{k,\alpha} \A_{k,\alpha}[e^k_{i}] - {\cal G}( p_k ,v_{k,\alpha}^{i} ) + \grad f(p_k) \right\|
	+ \left\| {\cal G}( p_k ,v_{k,\alpha}^{i} ) - \grad f(p_k) -\hess f(p_k)[ v_{k,\alpha}^{i} ] \right\| \nonumber\\
	& \leq  \left\| \sum_{j=1}^{n} \left\langle h_{k,\alpha} \A_{k,\alpha}[e^k_{i}] - {\cal G}( p_k ,v_{k,\alpha}^{i} ) + \grad f(p_k) , e^k_{j}  \right\rangle e^k_{j}  \right\|
	+ \frac{L}{2} \| v_{k,\alpha}^{i} \|^2 \nonumber \\
	& \leq \sqrt{n}  \max_{j=1,\ldots,n}  \left| \left\langle h_{k,\alpha} \A_{k,\alpha}[e^k_{i}] - {\cal G}( p_k ,v_{k,\alpha}^{i} ) + \grad f(p_k) , e^k_{j}  \right\rangle  \right|
	+ \frac{L}{2} ( h_{k,\alpha})^2  \nonumber \\
	& = \frac{ \sqrt{n} }{ h_{k,\alpha} }  \max_{j=1,\ldots,n}  \left| \underbrace{ (h_{k,\alpha})^2 \left\langle  \A_{k,\alpha}[e^k_{i}]  , e^k_{j}  \right\rangle - \left\langle {\cal G}( p_k ,v_{k,\alpha}^{i} ) , v_{k,\alpha}^{j}  \right\rangle
	+ \left\langle  \grad f(p_k) , v_{k,\alpha}^{j}  \right\rangle }_{ (i) } \right|
	+ \frac{L}{2} ( h_{k,\alpha} )^2.
	\end{align}
	Note that, as a consequence of \eqref{eq:mappauxGc1c2} and \eqref{eq:cG-exp.c1ec2}, the expression $(i)$ can be rephrased as
	\begin{align*}
	\MoveEqLeft
	(h_{k,\alpha})^2 \left\langle  \A_{k,\alpha}[e^k_{i}]  , e^k_{j}  \right\rangle - \left\langle \nabla\hat{f}_k( h_{k,\alpha}e^k_{i} ) , h_{k,\alpha} e_{j}^{k} \right\rangle
	+ \left\langle  \nabla\hat{f}_k(0) , h_{k,\alpha} e_{j}^{k}  \right\rangle  \\
	& = \left[ \hat{f}_k( h_{k,\alpha} e^{k}_{i} + h_{k,\alpha} e^{k}_{j} )
		- \hat{f}_k( h_{k,\alpha} e^{k}_{i} )  - \left\langle \nabla\hat{f}_k( h_{k,\alpha}e^k_{i} ) , h_{k,\alpha} e_{j}^{k} \right\rangle - \frac{1}{2}\left\langle \nabla^2\hat{f}_k( h_{k,\alpha}e^k_{i} )[h_{k,\alpha} e_{j}^{k}] , h_{k,\alpha} e_{j}^{k} \right\rangle \right]\\
	&\quad	- \left[ \hat{f}_k( h_{k,\alpha} e^{k}_{j} )
		- \hat{f}_k(0)
	- \left\langle  \nabla\hat{f}_k(0) , h_{k,\alpha} e_{j}^{k}  \right\rangle
	- \frac{1}{2}\left\langle  \nabla^2 \hat{f}_k(0)[h_{k,\alpha} e_{j}^{k}] , h_{k,\alpha} e_{j}^{k}  \right\rangle \right] \\
	& \quad +\frac{1}{2} \left\langle \left( \nabla^2 \hat{f}_k(h_{k,\alpha} e_{i}^{k} ) - \nabla^2 \hat{f}_k(0) \right)[h_{k,\alpha} e_{j}^{k}] , h_{k,\alpha} e_{j}^{k}  \right\rangle,\\
	\end{align*}
	for case 1, and as
	\begin{align*}
	\MoveEqLeft
	(h_{k,\alpha})^2 \left\langle  \A_{k,\alpha}[e^k_{i}]  , e^k_{j}  \right\rangle - \left\langle \p_{ v_{k,\alpha}^{i} }^{-1} \grad f(q_{k,\alpha}^{i} ) , v_{k,\alpha}^{j}  \right\rangle
	+ \left\langle  \grad f(p_k) , v_{k,\alpha}^{j}  \right\rangle & = f \left( \exp_{q_{k,\alpha}^{i} } \left(  \p_{ v_{k,\alpha}^{i} } v_{k,\alpha}^{j}  \right) \right)
		- f( q_{k,\alpha}^{i} )
		- f( q_{k,\alpha}^{j} )
		+f(p_k)
	- \left\langle \p_{ v_{k,\alpha}^{i} }^{-1} \grad f(q_{k,\alpha}^{i} ) , v_{k,\alpha}^{j}  \right\rangle
	+ \left\langle  \grad f(p_k) , v_{k,\alpha}^{j}  \right\rangle \\
		& = \left[ f \left( \exp_{q_{k,\alpha}^{i} } \left(  \p_{ v_{k,\alpha}^{i} } v_{k,\alpha}^{j}  \right) \right)
		- f( q_{k,\alpha}^{i} ) - \left\langle  \grad f( q_{k,\alpha}^{i} ) , \p_{ v_{k,\alpha}^{i} } v_{k,\alpha}^{j}  \right\rangle
		- \frac{1}{2} \left\langle \hess f(q_{k,\alpha}^{i})[\p_{ v_{k,\alpha}^{i} } v_{k,\alpha}^{j} ] , \p_{ v_{k,\alpha}^{i} } v_{k,\alpha}^{j}  \right\rangle \right] \\
		& \quad - \left[ f( q_{k,\alpha}^{j} )
		- f(p_k)
	- \left\langle  \grad f(p_k) , v_{k,\alpha}^{j}  \right\rangle -  \frac{1}{2} \left\langle \hess f( p_k )[ v_{k,\alpha}^{j} ] , v_{k,\alpha}^{j}  \right\rangle \right] \\
	       & \quad + \frac{1}{2} \left\langle \left( \p^{-1}_{ v_{k,\alpha}^{i} } \circ \hess f(q_{k,\alpha}^{i})\circ \p_{ v_{k,\alpha}^{i} }  - \hess f( p_k ) \right)v_{k,\alpha}^{j} ,  v_{k,\alpha}^{j}  \right\rangle,
	       	\end{align*}
	for case 2. Applying the modulus to both sides, utilizing the triangular inequality, considering that both $\nabla^2 \hat{f}_k$ and $\hess f$ are $L$-Lipschitz, and employing Lemma ~\ref{lem:ineq.hess.lips}, we obtain
	\begin{align*}
	\MoveEqLeft
	\left| (h_{k,\alpha})^2 \left\langle  \A_{k,\alpha}[e^k_{i}]  , e^k_{j}  \right\rangle - \left\langle \nabla\hat{f}_k( h_{k,\alpha}e^k_{i} ) , h_{k,\alpha} e_{j}^{k} \right\rangle
	+ \left\langle  \nabla\hat{f}_k(0) , h_{k,\alpha} e_{j}^{k}  \right\rangle \right|  \\
	& \leq \left| \hat{f}_k( h_{k,\alpha} e^{k}_{i} + h_{k,\alpha} e^{k}_{j} )
		- \hat{f}_k( h_{k,\alpha} e^{k}_{i} )  - \left\langle \nabla\hat{f}_k( h_{k,\alpha}e^k_{i} ) , h_{k,\alpha} e_{j}^{k} \right\rangle - \frac{1}{2}\left\langle \nabla^2\hat{f}_k( h_{k,\alpha}e^k_{i} )[h_{k,\alpha} e_{j}^{k}] , h_{k,\alpha} e_{j}^{k} \right\rangle \right| \\
	&\quad	+ \left| \hat{f}_k( h_{k,\alpha} e^{k}_{j} )
		- \hat{f}_k(0)
	- \left\langle  \nabla\hat{f}_k(0) , h_{k,\alpha} e_{j}^{k}  \right\rangle
	- \frac{1}{2}\left\langle  \nabla^2 \hat{f}_k(0)[h_{k,\alpha} e_{j}^{k}] , h_{k,\alpha} e_{j}^{k}  \right\rangle \right| \\
	& \quad + \frac{1}{2} \left|  \left\langle \left( \nabla^2 \hat{f}_k(h_{k,\alpha} e_{i}^{k} ) - \nabla^2 \hat{f}_k(0) \right)[h_{k,\alpha} e_{j}^{k}] , h_{k,\alpha} e_{j}^{k}  \right\rangle \right| \\
	& \leq \frac{L}{6}\| h_{k,\alpha} e_{j}^{k} \|^3 + \frac{L}{6}\| h_{k,\alpha} e_{j}^{k} \|^3
	+ \frac{1}{2} \left\| \nabla^2 \hat{f}_k(h_{k,\alpha} e_{i}^{k} ) - \nabla^2 \hat{f}_k(0) \right\|_{\op} \|h_{k,\alpha} e_{j}^{k} \| \| h_{k,\alpha} e_{j}^{k}  \|  \\
	& \leq  \frac{L}{3}\| h_{k,\alpha} e_{j}^{k} \|^3
	+ \frac{L}{2}  \| h_{k,\alpha} e_{j}^{k}  \|^3= \frac{5L}{6} ( h_{k,\alpha} )^3
	\end{align*}
	and
		\begin{align*}
	\MoveEqLeft
	\left| (h_{k,\alpha})^2 \left\langle  \A_{k,\alpha}[e^k_{i}]  , e^k_{j}  \right\rangle - \left\langle \p_{ v_{k,\alpha}^{i} }^{-1} \grad f(q_{k,\alpha}^{i} ) , v_{k,\alpha}^{j}  \right\rangle
	+ \left\langle  \grad f(p_k) , v_{k,\alpha}^{j}  \right\rangle \right|  \\
	       & \leq \left| f \left( \exp_{q_{k,\alpha}^{i} } \left(  \p_{ v_{k,\alpha}^{i} } v_{k,\alpha}^{j}  \right) \right)
		- f( q_{k,\alpha}^{i} ) - \left\langle  \grad f( q_{k,\alpha}^{i} ) , \p_{ v_{k,\alpha}^{i} } v_{k,\alpha}^{j}  \right\rangle
		- \frac{1}{2} \left\langle \hess f(q_{k,\alpha}^{i})[\p_{ v_{k,\alpha}^{i} } v_{k,\alpha}^{j} ] , \p_{ v_{k,\alpha}^{i} } v_{k,\alpha}^{j}  \right\rangle \right| \\
		& \quad + \left| f( q_{k,\alpha}^{j} )
		- f(p_k)
	- \left\langle  \grad f(p_k) , v_{k,\alpha}^{j}  \right\rangle -  \frac{1}{2} \left\langle \hess f( p_k )[ v_{k,\alpha}^{j} ] , v_{k,\alpha}^{j}  \right\rangle \right| \\
	       & \quad + \frac{1}{2} \left|  \left\langle \left( \p^{-1}_{ v_{k,\alpha}^{i} } \circ \hess f(q_{k,\alpha}^{i})\circ \p_{ v_{k,\alpha}^{i} }  - \hess f( p_k ) \right)v_{k,\alpha}^{j} ,  v_{k,\alpha}^{j}  \right\rangle \right| \\
	       & \leq \frac{L}{6} \| \p_{ v_{k,\alpha}^{i} } v_{k,\alpha}^{j} \|^3  + \frac{L}{6} \| v_{k,\alpha}^{j} \|^3 	+ \frac{1}{2} \left\| \p^{-1}_{ v_{k,\alpha}^{i} } \circ \hess f(q_{k,\alpha}^{i})\circ \p_{ v_{k,\alpha}^{i} }  - \hess f( p_k ) \right\|_{\op} \|v_{k,\alpha}^{j} \|^2 \\
	       & \leq \frac{L}{6} \| v_{k,\alpha}^{j} \|^3 + \frac{L}{6} \| v_{k,\alpha}^{j} \|^3 + \frac{L}{2} \| v_{k,\alpha}^{j} \|^3
	= \frac{5L}{6} (h_{k,\alpha})^3.
	\end{align*}
	By combining the above calculations with \eqref{eq:proisag} and  \eqref{eq:pro.reto(i)}, in any case, we have
	  $$
      \left\| \B_{k,\alpha}-\hess f(p_k)\right\|_{\op}  \leq
      \frac{  \sqrt{n}  }{ h_{k,\alpha} }  \left(   \frac{ \sqrt{n} }{ h_{k,\alpha} }   \frac{5L}{6} (h_{k,\alpha})^3
	+ \frac{L}{2} ( h_{k,\alpha} )^2  \right) = \frac{L( 5n + 3\sqrt{n} )}{6 \sigma_k }  \frac{\| v_{k-1} \|}{2^{\alpha -1} }.
      $$
      Therefore, the conclusion of the proof follows from
      Corollary~\ref{cor:gooddef}.	
\end{proof}

        \begin{remark}
        By using the orthonormality of the basis $\mathfrak{B}^k = \{ e^k_1,\ldots, e^k_n \}$ and \eqref{eq:algBsaac}, we obtain
\begin{align*}
\B_{k,\alpha} v & = \sum_{i=1}^{n} \langle \B_{k,\alpha} v, e_i \rangle e_i  \\
& = \sum_{i=1}^{n}  \sum_{j=1}^{n} \langle v, e_j \rangle \left\langle \B_{k,\alpha}  e_j  , e_i \right\rangle e_i \\
& = \frac{1}{2}  \sum_{i=1}^{n}  \sum_{j=1}^{n} \langle v, e_j \rangle  \left\langle  (\A_{k,\alpha} + \A^{*}_{k,\alpha} ) e_j , e_i \right\rangle e_i  \\
& = \frac{1}{2}  \sum_{i=1}^{n}  \sum_{j=1}^{n} \langle v, e_j \rangle \left(  \left\langle  \A_{k,\alpha} e_j , e_i \right\rangle +  \left\langle      \A_{k,\alpha} e_i , e_j \right\rangle \right) e_i,
\end{align*}
for all $v\in \T_{p_k} {\cal M}$. Hence, we can conclude that equation \eqref{eq:mappauxGc1c2} enables us to determine $\B_{k,\alpha}$ solely through evaluations of the function $f$.
\end{remark}

 \section{Numerical experiments}

In this section, the numerical performance of Algorithm 1 is illustrated by implementing the derivative-free form with the Manopt package \cite{Manopt}, where the parameters $\sigma_{1} = 1$ and $\theta = 1$, the approximated gradient $g_{k,\alpha}$ and approximated Hessian $B_{k,\alpha}$ are computed by \eqref{eq:algAdehgr} and \eqref{eq:mappauxGc1c2}, respectively. The Riemannian conjugate gradient method with Polak-Ribiere update formula was used to solve the cubic subproblem in Algorithm 1 \cite{Manopt}. The outer iteration terminates when $\|g_{k,\alpha}\|\leq 10^{-8}$. For the exactness of the derivative-free form of Algorithm 1 and comparison, the ARC method \cite{agarwal2021adaptive} is implemented as well. The cubic subproblem in the ARC method is also solved by using the Riemannian conjugate gradient method with Polak-Ribiere update formula. The ARC method terminates when $\|\grad f(p_{k})\|\leq 10^{-8}$. 
All the codes executions are carried out on a MacbookPro running macOS Ventura, 13.2.1, with 16 GB RAM, an Apple M1 Pro CPU, and Matlab R2022a. Additionally, to ensure reproducibility, the randomness is fixed by using the \texttt{rng(2024)} command.

We first consider six Riemannian optimization problems on six different Riemannian manifolds:
\begin{enumerate}
\item Top eigenvalue is to solve
\begin{equation*}
\max_{X\in \mathrm{Sp}(r)} \frac{1}{2}X^{T}AX,
\end{equation*}
where $A\in \mathbb{R}^{r\times r}$ is symmetric (randomly generated from i.i.d Gaussian entries) and $\mathrm{Sp}(r) = \{X|X\in\mathbb{R}^{r}, X^{T}X=1\}$ is a sphere manifold. Optimal objective value corresponds to the largest eigenvalue of $A$.

\item Dominant invariant subspace is to solve
\begin{equation*}
\max_{X\in \mathrm{Gr}(r,t)} \frac{1}{2}\mathrm{Tr}(X^{T}AX),
\end{equation*}
where $A\in \mathbb{R}^{r\times r}$ is symmetric (randomly generated from i.i.d Gaussian entries) and $\mathrm{Gr}(r,t) = \{\mathrm{span}(X): X\in\mathbb{R}^{r,t}, X^{T}X=I_{t}\} $ is a Grassmann manifold. Optima correspond to dominant invariant subspaces of A \cite{EdelmanAriasSmith1999}.

\item Elliptopt SDP is to solve
\begin{equation*}
\min_{X\in\mathrm{Ob}(r,t)} \frac{1}{2}\mathrm{Tr}(X^{T}AX),
\end{equation*}
where $A\in \mathbb{R}^{r\times r}$ is symmetric (randomly generated from i.i.d Gaussian entries) and $\mathrm{Ob}(r,t) = \{X|X\in\mathbb{R}^{r,t}, (XX^{T})_{ii} = 1, i=1,2,\cdots,r \}$ is an oblique manifold. The above problem is equivalent to the following SDP problem,
\begin{equation*}
\min_{Y\in\mathbb{R}^{r\times r}} \mathrm{AY}, \
\mathrm{s.t.}\ \mathrm{diag}(Y) = 1\ \mathrm{and}\ Y \ \mathrm{is\ positive\ semidefinite}.
\end{equation*}

\item  Truncated SVD is to solve
\begin{equation*}
\max_{U\in\mathrm{St}(r,t), V\in\mathrm{St}(s,t)} \mathrm{Tr}(U^{T}AVN),
\end{equation*}
where $A\in \mathbb{R}^{r,s}$ has i.i.d. random Gaussian entries, $N = \mathrm{diag}(t,t-1,\cdots,1)$ and $\mathrm{St}(r,t) = \{X|X\in\mathbb{R}^{r\times t}, X^{T}X = I_{t}\}$ is a Stiefel manifold. Global optima correspond to the $t$ dominant left and right singular vectors of $A$ \cite{sato2013riemannian}.

\item ShapeFit aims to reconstruct a rigid structure comprising $r$ point $x_{1},\cdots,x_{r}$ in $\mathbb{R}^{t}$ through a least-squares formulation. This reconstruction is derived from noisy measurements of selected pairwise directions $\frac{x_{i}-x_{j}}{\|x_{i}-x_{j}\|}$, where the pairs are uniformly chosen at random \cite{hand2018shapefit}. The set of points is centered and obeys one extra linear constraint to fix scaling ambiguity, so that the search space is a manifold, represent by $\mathrm{Sf}(t,r) = \{X|X\in\mathbb{R}^{t,r}, XE = 0, \mathrm{Tr}(M^{T}X) = 1\}$, where every entry of $E\in \mathbb{R}^{r}$ is $1$ and  $M\in\mathbb{R}^{t,r}$ is a given matrix. Refer \cite{Manopt} to construct the objective function and $M$.

\item Synchronization of rotations is to estimate $t$ rotation matrices $Q_{1}, \cdots , Q_{t}$ in the orthogonal group $\mathrm{SO}(r)$ from noisy relative measurements $H_{ij}\approx Q_{i}Q_{j}^{T}$ for an Erdos-Renyi random set of pairs $(i,j)$ following a maximum likelihood formulation \cite{boumal2013robust}. The details regarding the distribution of measurements and the associated objective function can be found in the referenced material. Additionally, the initial guess is utilzed as the reference to prevent convergence to an undesirable local optimum.
\end{enumerate}

For each problem, we create a single instance and generate a random initial guess, except for problem 6, which is initialized deterministically. Subsequently, we execute each algorithm starting from the same initial guess on that specific instance. From Table \ref{table_1}, it is evident that, for all the given problems, the norms of the approximated Riemannian gradients and Riemannian gradients produced by both Algorithm 1 and ARC are consistently smaller than $10^{-8}$. Notably, except for problem 3, Algorithm 1 and ARC yield identical objective function values, which affirms that Algorithm 1 successfully achieves a minimum for these five problems. To validate the optimality of the point obtained by Algorithm 1 for problem 3, we substitute it into ARC, revealing a corresponding Riemannian gradient of $1.2\times10^{-11}$. This comparison demonstrates that Algorithm 1 outperforms ARC in achieving a superior minimum for problem 3. Figures \ref{function_value} and \ref{norm_of_gradient} visually represent the objective function values and the norms of the approximated Riemannian gradients at each iterate across the six Riemannian optimization problems. The objective function values exhibit monotonic behavior on the tested problems, and as the point approaches the minimum, the norm of the approximated Riemannian gradient rapidly converges, aligning with the convergence theory of adaptive regularization with cubic.

\begin{table}[htbp]
\centering
\setlength{\tabcolsep}{1.5mm}{
\begin{tabular}{cccccccc}
\hline
\multicolumn{2}{c}{} & \multicolumn{3}{c}{Algorithm 1} &  \multicolumn{3}{c}{ARC}  \\
   \cline{3-8}
problem     & manifold     & OFV    & $\|g_{k,\alpha}\|$     & \#It           & OFV    & $\|\grad f(p_{k})\|$     & \#It             \\ \hline
1    & $\mathrm{Sp}(500)$       & 31.5403      & $6.4\times10^{-10}$     & 12      & 31.5403     & $6.8\times10^{-14}$ & 7            \\
2    & $\mathrm{Gr}(45,25)$      & 45.6298      & $4.6\times10^{-10}$     & 20      & 45.6298     & $3.4\times10^{-10}$ & 10            \\
3    & $\mathrm{Ob}(50,15)$       & -1.0700      & $1.7\times10^{-11}$     & 88      & -0.7442     & $8.3\times10^{-9}$ & 51           \\
4    & $\mathrm{St}(35,10)\times \mathrm{St}(25,10)$      & 447.2950      & $2.8\times10^{-9}$     & 42      & 447.2950     & $2.2\times10^{-13}$ & 15          \\
5    & $\mathrm{Sf}(3,500)$          & 5.5249      & $1.6\times10^{-11}$     & 3      & 5.5249     & $1.0\times10^{-12}$ & 3           \\
6    & $\underbrace{[\mathrm{SO}(3)]\times \cdots \times [\mathrm{SO}(3)]}_{50}$     & -3.3746      & $1.6\times10^{-9}$     & 21      & -3.3746     & $1.6\times10^{-9}$ & 6          \\ \hline \\
\end{tabular}}
\caption{The results of Algorithm 1 and ARC on six Riemannian optimization problems. OFV and \#It represent the objective function value and number of iterations, respectively. $\|g_{k,\alpha}\|$ is the norm of the approximated Riemannian gradient in Algorithm 1 and $\|\grad f(p_{k})\|$ is the norm of the Riemannian gradient in ARC.} \label{table_1}
\end{table}

 \begin{figure}[hthp]
\centering
\includegraphics[width=5.5in,height=5.0in]{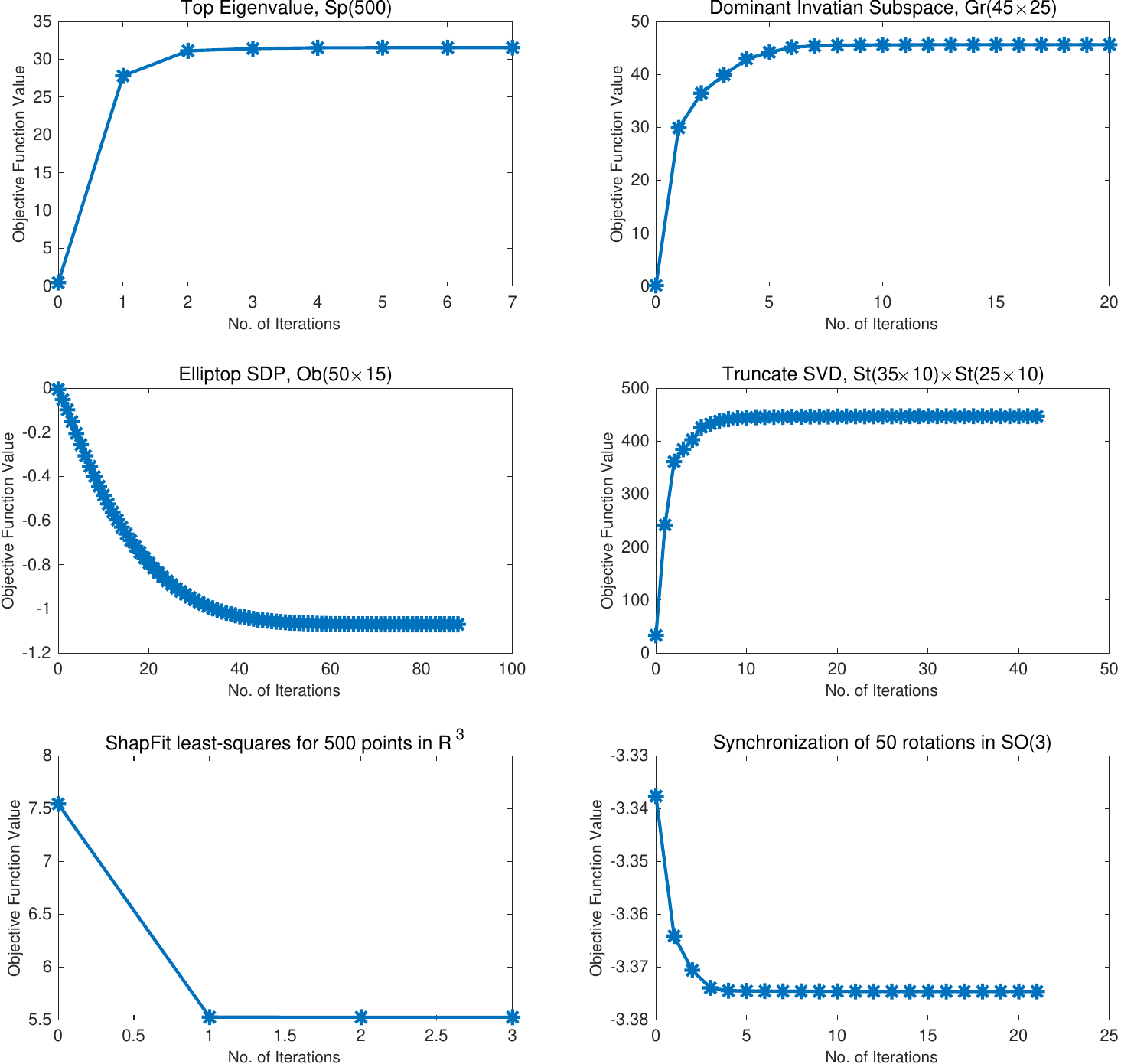}
\caption{Objective function value at each iteration on the six Riemannian optimization problems.}\label{function_value}
\end{figure}

\begin{figure}[hthp]
\centering
\includegraphics[width=5.5in,height=5.0in]{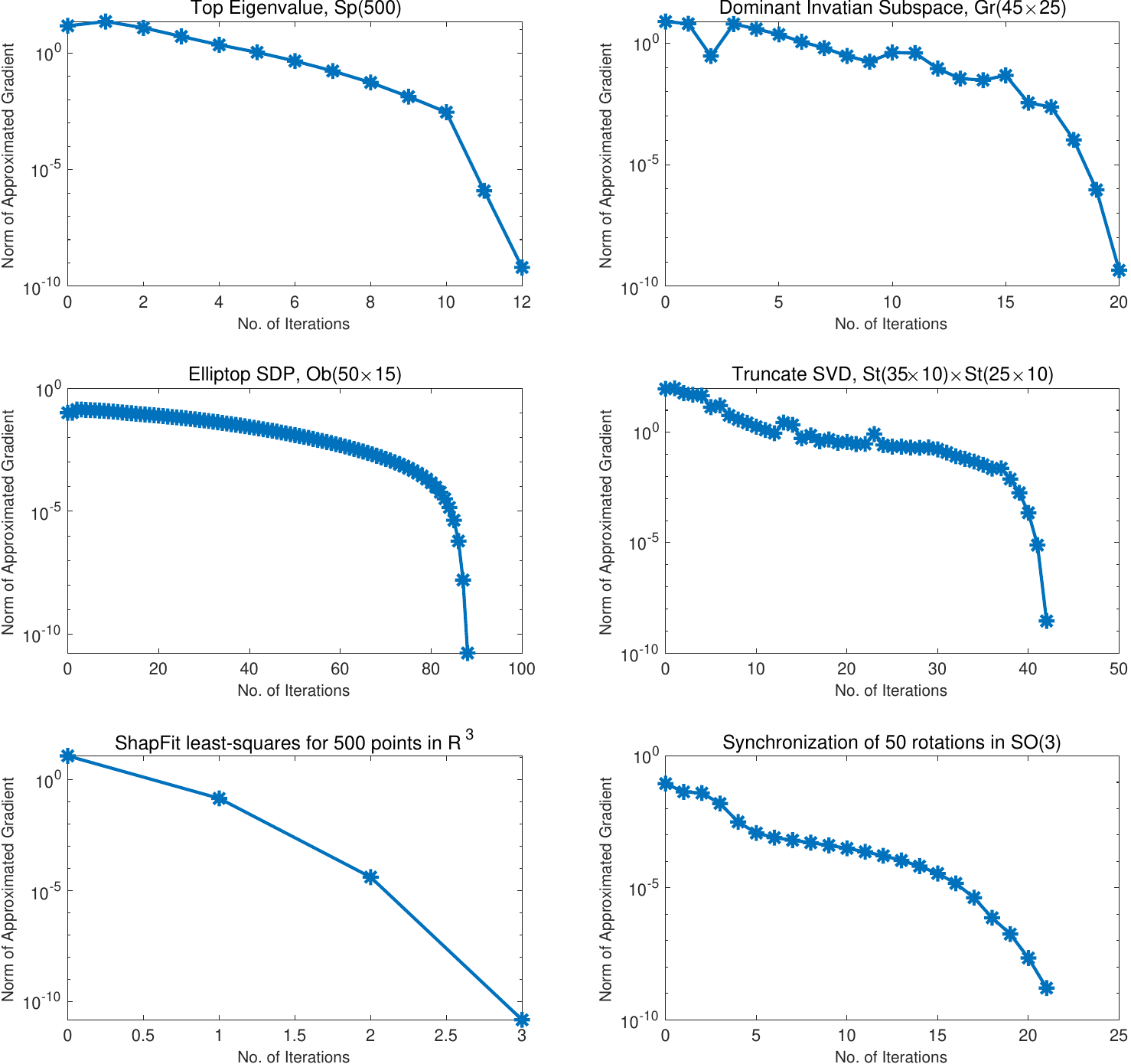}
\caption{Norm of approximated Riemannian gradient at each iteration on the six Riemannian optimization problems.}\label{norm_of_gradient}
\end{figure}

Next, we consider the following problem.
\begin{enumerate}
\item[7.] Minimizing a composite function on a sphere manifold is to solve
\begin{equation}\label{composite_function}
\min_{X\in \mathrm{Sp}(r_{1})} \frac{1}{r_{4}}\|  f_{3}\circ f_{2}\circ f_{1}(X)\|,
\end{equation}
where $f_{i}(X) = \mathrm{Swish}(A_{i}X+b_i)$, $A_{i}\in \mathbb{R}^{r_{i+1}\times r_{i}}, b_{i}\in\mathbb{R}^{r_{i+1}}, i=1,\cdots, 3$. Here, $(\mathrm{Swish}(b))_{j} = \mathrm{swish}(b_{j})$ and $\mathrm{swish}(x) = \frac{x}{1+\exp(-x)}$, namely, the function $\mathrm{Swish}$ is to use function $\mathrm{swish}$ to act on each component. Here, $A_{i}$ and $b_{i}$, $i=1,\cdots, 3$ are randomly generated from i.i.d Gaussian entries and $r_{i}, i=1, \cdots, 4$ are set as $50, 100, 90$ and $80$, respectively. 
\end{enumerate}
The objective function in problem 7 usually appears in the deep learning, which can be considered as a three-layer fully connected neural network. Here, the function $\mathrm{swish}$ is an active function \cite{searching2017}. Since the Riemannian gradient and hessian of the objective function in \eqref{composite_function} is hard to compute, we can not apply ARC. But our Algorithm 1 does not involve the exact Riemannian gradient and hessian, which can be employed to solve problem 7. After applying Algorithm 1, we obtain a point that corresponding objective function value is 0.1770 and norm of the approximated Riemannian gradient $\|g_{k,\alpha}\|$ is $4.0\times 10^{-9}$. Combing with the discussion of previous 6 problems, Algorithm 1 indeed gives a minimum for problem 7. In addition, Figure \ref{figure_composite_function} also displays the objective function value and norm of the approximated Riemannian gradient at each iteration for problem 7.

\begin{figure}[hthp]
\centering
\includegraphics[width=5.5in,height=2.0in]{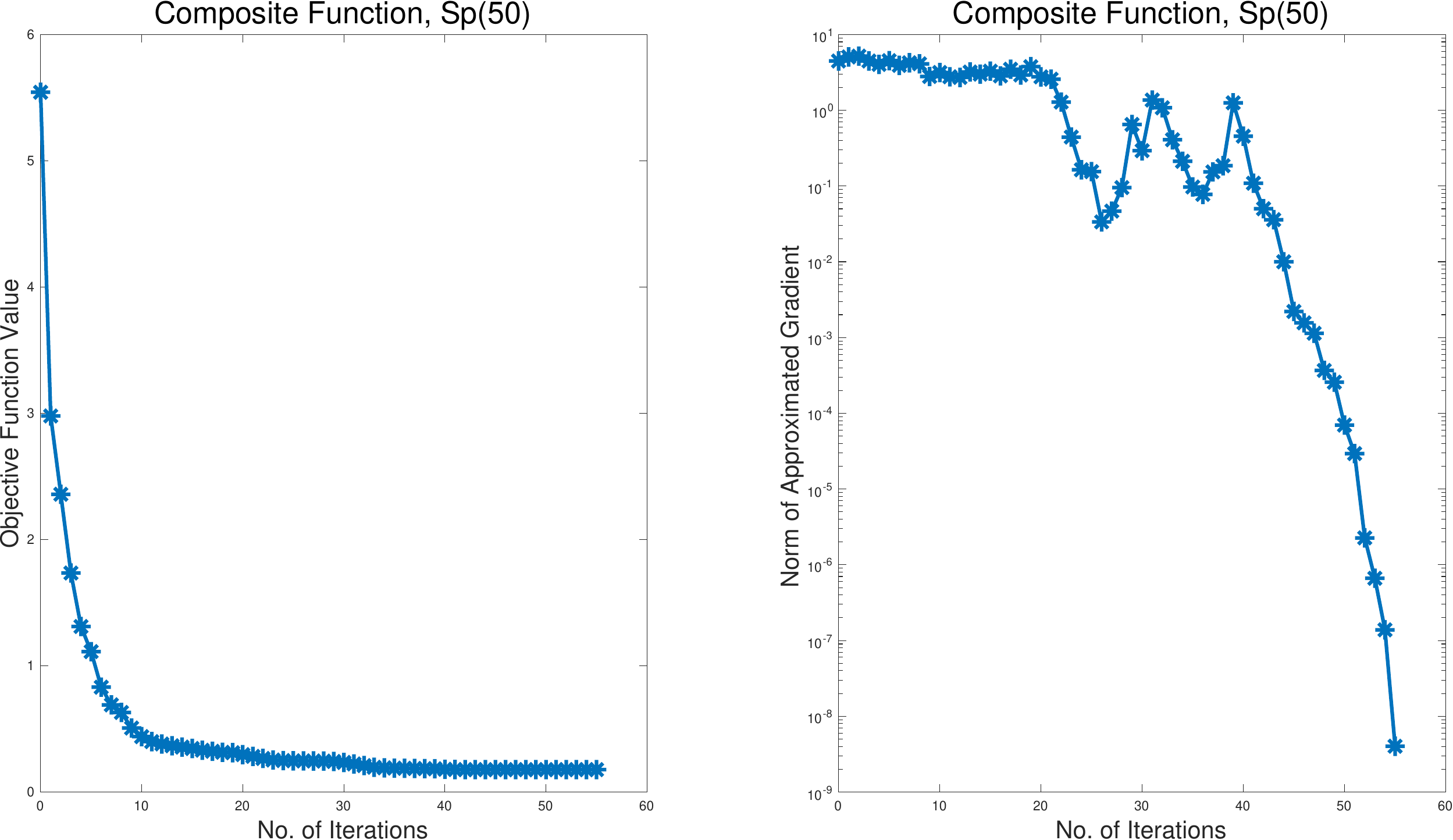}
\caption{Objective function value and norm of approximated Riemannian gradient at each iteration for minimizing the composite function on a sphere manifold.}\label{figure_composite_function}
\end{figure}

Finally, Algorithm 1, being a derivative-free method that includes the computation of orthonormal bases for the tangent space, typically requires more iterations and longer running time when the gradient and Hessian of the Riemannian optimization problem are easily computable. However, in cases where only the objective function value is known, Algorithm 1 proves to be a practical and effective choice.

        \section{Conclusion}

In this paper, we present a quasi-Newton algorithm with cubic regularization tailored specifically for Riemannian manifolds. The distinctive feature of this algorithm lies in its ability to operate without prior knowledge of the gradient and Hessian of the objective function. Instead, it only requires approximations that satisfy a condition analogous to those mandated by inexact algorithms. Importantly, we demonstrate that approximations obtained through finite-differences meet this condition, allowing us to assert that the algorithm can be considered derivative-free.

Regarding future research directions, we find the exploration and application of the inexactness outlined in this paper in algorithms with non-differentiable objective functions to be intriguing. This exploration could potentially lead to the development of a subdifferential-free algorithm, a prospect of particular significance given the often challenging nature of calculating a subgradient.


\def\cprime{$'$}


\begin{thebibliography}{10}

\bibitem{AbsiBakerGallivan2007}
P.-A. Absil, C.~G. Baker, and K.~A. Gallivan.
\newblock Trust-region methods on {R}iemannian manifolds.
\newblock {\em Foundations of Computational Mathematics}, 7(3):303--330, 2007.

\bibitem{absil2008optimization}
P.-A. Absil, R.~Mahony, and R.~Sepulchre.
\newblock {\em Optimization algorithms on matrix manifolds}.
\newblock Princeton University Press, 2008.

\bibitem{agarwal2021adaptive}
N.~Agarwal, N.~Boumal, B.~Bullins, and C.~Cartis.
\newblock Adaptive regularization with cubics on manifolds.
\newblock {\em Mathematical Programming, Series A}, 188(1):85--134, 2021.

\bibitem{birgin2019}
E.~G. Birgin, N.~Krejić, and J.~M. Martínez.
\newblock Iteration and evaluation complexity for the minimization of functions
  whose computation is intrinsically inexact.
\newblock {\em Mathematics of Computation}, 89:253--278, 2020.

\bibitem{boumal2023intromanifolds}
N.~Boumal.
\newblock {\em An introduction to optimization on smooth manifolds}.
\newblock Cambridge University Press, 2023.

\bibitem{Manopt}
N.~Boumal, B.~Mishra, P.-A. Absil, and R.~Sepulchre.
\newblock Manopt, a matlab toolbox for optimization on manifolds.
\newblock {\em Journal of Machine Learning Research}, 15:1455--1459, 2014.

\bibitem{boumal2013robust}
N.~Boumal, A.~Singer, and P.-A. Absil.
\newblock Robust estimation of rotations from relative measurements by maximum
  likelihood.
\newblock In {\em IEEE 52nd Annual Conference on Decision and Control (CDC)},
  pages 1156--1161, 2013.

\bibitem{CartisToint2011}
C.~Cartis, N.~Gould, and P.~Toint.
\newblock Adaptive cubic regularisation methods for unconstrained optimization.
  part i: motivation, convergence and numerical results.
\newblock {\em Mathematical Programming, Series A}, 127:245--295, 2011.

\bibitem{CartisToint20112}
C.~Cartis, N.~Gould, and P.~Toint.
\newblock Adaptive cubic regularisation methods for unconstrained optimization.
  part ii: worst-case function and derivative evaluation complexity.
\newblock {\em Mathematical Programming, Series A}, 130:295--319, 2011.

\bibitem{cartis2012oracle}
C.~Cartis, N.~I. Gould, and P.~L. Toint.
\newblock On the oracle complexity of first-order and derivative-free
  algorithms for smooth nonconvex minimization.
\newblock {\em SIAM Journal on Optimization}, 22(1):66--86, 2012.

\bibitem{DengMu2023}
Y.~Deng and T.~Mu.
\newblock Faster {R}iemannian {N}ewton-type optimization by subsampling and
  cubic regularization.
\newblock {\em Machine Learning}, 112:3527--3589, 2023.

\bibitem{DoCa92}
M.~P. do~Carmo.
\newblock {\em Riemannian geometry}.
\newblock Mathematics: Theory \& Applications. Birkh\"auser Boston, Inc.,
  Boston, MA, 1992.
\newblock Translated from the second Portuguese edition by Francis Flaherty.

\bibitem{EdelmanAriasSmith1999}
A.~Edelman, T.~A. Arias, and S.~T. Smith.
\newblock The geometry of algorithms with orthogonality constraints.
\newblock {\em SIAM Journal on Matrix Analysis and Applications},
  20(2):303--353, 1999.

\bibitem{grapiglia2022cubic}
G.~N. Grapiglia, M.~L. Gon{\c{c}}alves, and G.~Silva.
\newblock A cubic regularization of {N}ewton’s method with finite difference
  hessian approximations.
\newblock {\em Numerical Algorithms}, 90(2):607--630, 2022.

\bibitem{hand2018shapefit}
P.~Hand, C.~Lee, and V.~Voroninski.
\newblock Shapefit: Exact location recovery from corrupted pairwise directions.
\newblock {\em Communications on Pure and Applied Mathematics}, 71(1):3--50,
  2018.

\bibitem{kholer2017}
J.~M. Kohler and A.~Lucchi.
\newblock Sub-sampled cubic regularization for non-convex optimization.
\newblock {\em Proceedings of the 34th international conference on machine
  learning}, 70.

\bibitem{Lee:2003:1}
J.~M. Lee.
\newblock {\em Introduction to Smooth Manifolds}, volume 218 of {\em Graduate
  Texts in Mathematics}.
\newblock Springer-Verlag, New York.

\bibitem{lee2006riemannian}
J.~M. Lee.
\newblock {\em Riemannian manifolds: an introduction to curvature}, volume 176.
\newblock Springer Science \& Business Media, 2006.

\bibitem{nesterovpolyak2006}
Y.~Nesterov and B.~T. Polyak.
\newblock Cubic regularization of {N}ewton method and its global performance.
\newblock {\em Mathematical Programming}, 108:177--205, 2006.

\bibitem{nocedal1999numerical}
J.~Nocedal and S.~J. Wright.
\newblock {\em Numerical optimization}.
\newblock Springer, 1999.

\bibitem{searching2017}
P.~Ramachandran, B.~Zoph, and Q.~V. Le.
\newblock Searching for activation functions.
\newblock {\em arXiv:1710.05941}, 2017.

\bibitem{sato2013riemannian}
H.~Sato and T.~Iwai.
\newblock A {R}iemannian optimization approach to the matrix singular value
  decomposition.
\newblock {\em SIAM Journal on Optimization}, 23(1):188--212, 2013.

\bibitem{Tripuraneni2018}
N.~Tripuraneni, M.~Stern, C.~Jin, J.~Regier, and M.~I. Jordan.
\newblock Stochastic cubic regularization for fast nonconvex optimization.
\newblock {\em Advances in Neural Information Processing Systems}, 2018.

\bibitem{Tu:2011:1}
L.~W. Tu.
\newblock {\em An Introduction to Manifolds}.
\newblock Universitext. Springer, New York, 2 edition.

\bibitem{WangLan2019}
Z.~Wang, Y.~Zhou, Y.~Liang, and G.~Lan.
\newblock A note on inexact gradient and {H}essian conditions for cubic
  regularized {N}ewton’s method.
\newblock {\em Operations Research Letters}, 47(2):146--149, 2019.

\bibitem{Weiwei2013}
W.~Yang, Y.~Yang, C.~Zhang, and M.~Cao.
\newblock A {N}ewton-like trust region method for large-scale unconstrained
  nonconvex minimization.
\newblock {\em Abstract and Applied Analysis}, 2013.

\bibitem{ZhangZhang2018}
J.~Zhang and S.~Zhang.
\newblock A cubic regularized {N}ewton’s method over {R}iemannian manifolds.
\newblock {\em arXiv:1805.05565}.

\end{thebibliography}

\end{document}